\documentclass[11pt,reqno]{amsart}
\usepackage{amsmath}
\usepackage{amssymb}
\usepackage{amsthm}
\usepackage{mathtools}
\usepackage{color}
\usepackage{eucal}
\usepackage{tikz}
\usepackage{gastex}
\usepackage{stmaryrd}
\usepackage{caption,subcaption}
\usepackage{latexsym}
\usepackage{indentfirst}
\usepackage{graphicx,accents}

\usetikzlibrary{snakes}

\DeclareSymbolFont{rsfscript}{OMS}{rsfs}{m}{n}
\DeclareSymbolFontAlphabet{\mathrsfs}{rsfscript}

\numberwithin{equation}{section}
\newtheorem{prop}{Proposition}[section]
\newtheorem{teor}[prop]{Theorem}
\newtheorem{lem}[prop]{Lemma}
\newtheorem{cor}[prop]{Corollary}

\def\Jc{\mathrel{\mathrsfs{J}}}
\def\Dc{\mathrel{\mathrsfs{D}}}
\def\Hc{\mathrel{\mathrsfs{H}}}
\def\Lc{\mathrel{\mathrsfs{L}}}
\def\Rc{\mathrel{\mathrsfs{R}}}
\def\Kc{\mathrel{\mathrsfs{K}}}
\def\Rcc{\mathrm{R}}

\def\Lcc{\mathrm{L}}
\def\Dcc{\mathrm{D}}

\def\Kcc{\mathrm{K}}

\def\ep{\epsilon} 
\def\up{\upsilon} 
\def\la{\lambda} 
\def\ka{\kappa} 
\def\be{\beta} 
\def\si{\sigma} 

\def\ap{\approx}
\def\xr{\xrightarrow}
\def\cev{\overset{{}_{\shortleftarrow}}}
\def\cevl{\overset{{}_{\longleftarrow}}}
\def\cevm{\overset{{}_{\leftarrow}}}
\def\ol{\overline}

\renewcommand{\iff}{if and only if}

\DeclareMathOperator{\ft}{FI}
\DeclareMathOperator{\fo}{F_1}
\DeclareMathOperator{\foo}{F_1^1}
\DeclareMathOperator{\fot}{F_1^\circ}
\DeclareMathOperator{\La}{L}
\DeclareMathOperator{\mrx}{MR}
\DeclareMathOperator{\mx}{M}
\DeclareMathOperator{\mi}{i-M}
\DeclareMathOperator{\mri}{i-MR}
\DeclareMathOperator{\G}{G}
\DeclareMathOperator{\Ho}{H}

\DeclareMathOperator{\R}{R}

\def\ftx{\ft(X)}
\def\ftt{\ft_2}

\def\Gp{\G^+}
\def\Hx{\Ho}
\def\Hxp{\Ho^+}

\def\Hix{\Ho_i}
\def\mxm{\mx^1}
\def\lx{\La}
\def\lo{\La_1}
\def\lop{\La_1^+}
\def\lom{\La_1^-}
\def\ltp{\La_2^+}
\def\ltm{\La_2^-}
\def\lt{\La_2}

\def\pres{\langle\G,\R\rangle}

\title{Regular semigroups weakly generated by one element}
\author{Lu\'\i s Oliveira}
\address{CMUP, Departamento de Matem\'atica,
Faculdade de Ci\^encias, Universidade do Porto,
rua do Campo Alegre s/n, 4169-007 Porto, Portugal}
\email{loliveir@fc.up.pt}

\begin{document}

\begin{abstract}
In this paper we study the regular semigroups weakly generated by a single element $x$, that is, with no proper regular subsemigroup containing $x$. We show there exists a regular semigroup $\fo$ weakly generated by $x$ such that all other regular semigroups weakly generated by $x$ are homomorphic images of $\fo$. We define $\fo$ using a presentation where both sets of generators and relations are infinite. Nevertheless, the word problem for this presentation is decidable. We describe a canonical form for the congruence classes given by this presentation, and explain how to obtain it. We end the paper studying the structure of $\fo$. In particular, we show that the `free regular semigroup $\ftt$ weakly generated by two idempotents' is isomorphic to a regular subsemigroup of $\fo$ weakly generated by $\{xx',x'x\}$.
\end{abstract} 


\maketitle

\section{Introduction}

An element $s$ of a semigroup $S$ is called \emph{regular} if $sts=s$ for some $t\in S$. Note that $s'=tst$ is a (von Neumann) \emph{inverse} of $s$, that is, $ss's=s$ and $s'ss'=s'$. We denote by $V(s)$ the set of all inverses of $s$ in $S$. A \emph{regular semigroup} is a semigroup with all elements regular. As usual, we denote by $E(S)$ the set of idempotents of $S$. 


A subset of a regular semigroup $S$ may not generate a regular subsemigroup. Hence, the following notion of ``weakly generated'' seems natural: a regular subsemigroup $T$ of $S$ is weakly generated by a subset $X$ if $T$ has no proper regular subsemigroup containing $X$. Of course, this does not mean that $T$ is effectively generated by $X$. Very often, the subsemigroup generated by $X$ is a proper non-regular subsemigroup of $T$. In particular, we say that $S$ is weakly generated by $X$ if $S$ has no proper regular subsemigroup containing $X$. We should also point out that the same set $X$ may weakly generate several distinct regular subsemigroups of $S$.

The interest in studying the structure of regular semigroups weakly generated by a set $X$ came first from the theory of e-varieties of regular semigroups (see \cite{ha1,ks1} for the definition of e-variety). The existence of bifree objects on an e-variety $\bf V$ depends on the validity of the following property in $\bf V$: for all $S\in\bf V$ and all matched subset $A$ of $S$, there exists a unique regular subsemigroup of $S$ weakly generated by $A$ (see \cite{yeh}).

In \cite{LO22} we studied the structure of the regular semigroups weakly generated by idempotents. We proved the existence of a \emph{free regular semigroup $\ftx$ weakly generated by a set $X$ of idempotents} in the sense that $\ftx$ is weakly generated by $X$ and all other regular semigroups weakly generated by $X$ are homomorphic images of $\ftx$. We got $\ftx$ by introducing a presentation with both sets of generators and relations infinite. Despite that fact, the word problem for that presentation is decidable since a procedure to obtain a canonical representative for each congruence class was described. 

If $\ft_k$ denotes $\ftx$ for $|X|=k$, we proved that $\ftt$ contains copies of all $\ft_k$ as subsemigroups. Thus, any regular semigroup generated by a finite set of idempotents must strongly divide $\ftt$, that is, must be a homomorphic image of some regular subsemigroup of $\ftt$. The case $\ft_1$ is of no interest since any regular semigroup weakly generated by a single idempotent is just the trivial semigroup. 

We intend to look now to the general non-idempotent case. In other words, we plan to see if there exists a regular semigroup $F(X)$ weakly generated by a set $X$ (not necessarily of idempotents) such that all other regular semigroups weakly generated by $X$ are homomorphic images of $F(X)$. The case where $X$ is a singleton set is now distinct from the idempotent case: there are nontrivial regular semigroups weakly generated by a single element. So, in this paper, we focus only on the case where $X$ is a singleton set and prove that $F(X)$ exists for $|X|=1$. 

This paper will follow the structure of \cite{LO22}. In fact, we will use the same terminology, but the concepts will be more complex here. For example, terms such as landscape and mountain will be used here again but to refer to words with a more intricate structure. The results we will present are also similar to the ones of \cite{LO22}, and some of the proofs use even the same arguments (or with just small adaptations), although they refer to more complex objects. For the sake of completeness, we decided to include those proofs again in this paper.

We will denote $F(X)$ by $\fo$ for $|X|=1$, and we will represent by $x$ the unique element of $X$. In fact, we will usually not mention $X$ but its unique element $x$ instead. As stated above, the main goal of this paper is to prove the existence of $\fo$. We begin by introducing $\fo$ using a presentation similarly to \cite{LO22}. It is an interesting exercise to compare the presentation for $\fo$ with the presentation for $\ftx$. We will see their resemblance, but also notice a new `ingredient' appearing in $\fo$, namely conjugation. We then prove that the presentation introduced here has decidable word problem by obtaining a canonical form for each congruence class. To attest the richness of the structure of $\fo$, we will prove that $\fo$ has a regular subsemigroup weakly generated by the idempotents $\{xx',x'x\}$ isomorphic to $\ftt$, where $x'$ is the only inverse of $x$ in $\fo$. Since all regular semigroups weakly generated by a finite set of idempotents strongly divide $\ftt$, they also strongly divide $\fo$.

This paper is organized as follows. In the next section, we recall some basic concepts needed for this paper. We include also a brief description of the construction of $\ftt$ since we will need it later. We change also the terminology used in \cite{LO22} so that it does not overlap with the terminology used here. In Section 3 we introduce a presentation and define $\fo$. We present also a solution for the word problem for this presentation. We prove that $\fo$ is a regular semigroup weakly generated by $x$ in Section 4. To achieve this result we need to characterize first the Green's relations on $\fo$. Since we characterize the Green's relations in this section, we end it by describing also the idempotents, the inverses and the natural partial order. However, as it will become evident, this description will be just theoretical and not very useful for practical purposes. In Section 5, we show that all regular semigroups weakly generated by $x$ are homomorphic images of $\fo$. Finally, in the last section, we prove that $\fo$ has a regular subsemigroup weakly generated by $\{xx',x'x\}$ isomorphic to $\ftt$; thus concluding that all regular semigroups generated by a finite set of idempotents strongly divide $\fo$.

\section{Preliminaries}\label{sec2}

Given a semigroup $S$, $S^1$ denotes the monoid obtained by adding an identity element to $S$ if necessary. The quasi-orders $\leq_{\Lc}$,  $\leq_{\Rc}$ and $ \leq_{\Jc}$ on $S$ are defined as follows:
$$s\leq_{\Lc} t\Leftrightarrow S^1s\subseteq S^1t,\quad s\leq_{\Rc} t\Leftrightarrow sS^1\subseteq tS^1,\quad s\leq_{\Jc} t\Leftrightarrow S^1sS^1\subseteq S^1tS^1.$$
Thus, the Green's equivalence relations $\Lc$, $\Rc$ and $\Jc$ are just the relations
$$\Lc\,=\,\leq_{\Lc}\cap\geq_{\Lc},\quad \Rc\,=\,\leq_{\Rc}\cap\geq_{\Rc}\quad\mbox{ and }\quad \Jc\,=\,\leq_{\Jc}\cap\geq_{\Jc}\,,$$
where $\geq_{\Lc}$, $\geq_{\Rc}$ and $\geq_{\Jc}$ are respectively the dual relations of  $\leq_{\Lc}$, $\leq_{\Rc}$ and $\leq_{\Jc}$. Let $\Hc\,=\,\Lc\cap\Rc$ and $\Dc\,=\,\Lc\vee\Rc$ be the other two Green's relations. For a regular semigroup $S$, we can use $S$ instead of $S^1$ in the previous definitions. We will denote by $\Kcc_a$ the $\Kc$-class of the element $a\in S$, for $\Kc\in\{\Hc,\Lc,\Rc,\Dc,\Jc\}$. 

A crucial concept for this paper is also the notion of \emph{sandwich set} $S(e,f)$ of two idempotents $e$ and $f$ of $S$. This concept has a few equivalent definitions. We list some of them next:
$$\begin{array}{rl}
S(e,f)\hspace*{-.2cm}&=V(ef)\cap E(fSe)=fV(ef)e\\ [.2cm]
&=\{g\in E(S)\,|\; fg=g=ge \mbox{ and } egf=ef\}\,.
\end{array}$$
The sandwich set $S(e,f)$ is nonempty precisely when $ef$ is a regular element of $S$. Thus, $S(e,f)$ is always nonempty if $S$ is regular. In fact, $S(e,f)$ is always a rectangular band and a subsemigroup of $S$, whenever it is nonempty. There is another interesting property of these sets: if $e,e_1,f,f_1\in E(S)$ are such that $e\Lc e_1$ and $f\Rc f_1$, then $S(e,f)=S(e_1,f_1)$. Hence, for regular semigroups, one can extend the definition of sandwich set to all elements of $S$ as follows: for $a,b\in S$, let $S(a,b)=S(a'a,bb')$ for some (any) $a'\in V(a)$ and $b'\in V(b)$.

If $S$ is a regular semigroup, there is also another important relation on $S$,  the \emph{natural partial order} $\leq\,$:
$$s\leq t\quad\Leftrightarrow\quad  s=et=tf \;\mbox{ for some } e,f\in E(S)\,.$$
A useful fact about the natural partial order is that we can choose the idempotents $e$ and $f$ such that $e\Rc s\Lc f$. 


Let $X$ be a nonempty set. As usual, we denote by $X^+$ the free semigroup on $X$. The elements of $X$ are called \emph{letters} in this context, while the elements of $X^+$ are called \emph{words}. The \emph{content} of a word $u$ is the set of letters from $X$ that occur in $u$, and the \emph{length} of $u$ is the number of letters (counting repetitions) that occur in $u$. For $u\in X^+$, let $\si(u)$ and $\tau(u)$ denote the first and the last letter of $u$, respectively. 

Next, we recall the construction of $\ftt$ introduced in \cite{LO22}\label{hfi} for any set $X$. As explained before, we need to modify the notation and terminology used in \cite{LO22} so that it does not overlap with the one used in this paper. Hence $h$, $H$ and $\varrho$ will replace the symbols $g$, $G$ and $\rho$ used in \cite{LO22}. Given a triple $h=(a,b,c)$, we denote $a$, $b$ and $c$ by $h^r$, $h^c$ and $h^l$, respectively. 

Let $X=\{e,f\}$ be a two element set, and consider
$$\Ho_0=\{1\}\quad\mbox{ and }\quad\Ho_1=\{(1,e,1),\,(1,f,1)\},$$
where $1$ is a new symbol not in $X$. We identify each $x\in X$ with the triple $(1,x,1)\in \Ho_1$. Thus $x^r=x^l=1$. Recursively for $i>1$, assume that $\Ho_{i-1}$ is a set of triples and let $\overline{\Ho}_i=\Ho_{i-1}\times \Ho_{i-2}\times \Ho_{i-1}$ and
$$\Hix=\left\{h\in \overline{\Ho}_i\,|\; h^l\neq h^r,\; h^c\in\{(h^l)^l,(h^l)^r\}\cap\{(h^r)^l,(h^r)^r\}\right\}\,.$$ 
Note that 
$$\Ho_2=\{h_{2,1},\,h_{2,2}\}$$ 
for $h_{2,1}=(e,1,f)$ and $h_{2,2}=(f,1,e)$, and
$$\Ho_3=\{h_{3,1},\,h_{3,2},\,h_{3,3},\,h_{3,4}\}$$ 
for $h_{3,1}=(h_{2,1},e,h_{2,2})$, $h_{3,2}=(h_{2,1},f,h_{2,2})$, $h_{3,3}=(h_{2,2},e,h_{2,1})$ and $h_{3,4}=(h_{2,2},f,h_{2,1})$. Let $\Ho=\cup_{i\in\mathbb{N}_0}\Hix$. 

Consider the following two relations on the free semigroup $\Hxp$: 
$$\varrho_e=\{(1h,h),(h1,h),(h^2,h)\,|\;h\in \Hx\}$$
and
$$\varrho_s=\left\{(h^ch^lh,h),(hh^rh^c,h),(h^rh^chh^ch^l,h^r h^ch^l)\,|\; h\in \Hix \mbox{ with } i\geq 2 \right\}.$$
Let $\ft^1_2$ be the semigroup given by the presentation $\langle \Hx,\,\varrho_e\cup \varrho_s\rangle$, that is, $\ft^1_2=\Hxp/\varrho$ where $\varrho$\label{varrho} is the smallest congruence on $\Hxp$ containing $\varrho_e\cup\varrho_s$. Note that $\varrho_e$ just tells us that $\ft^1_2$ is an idempotent generated monoid with identity element $1\varrho$. Using both $\varrho_e$ and $\varrho_s$, one can show easily that $(h^rh^c)\varrho$ and $(h^ch^l)\varrho$ are idempotents of $\ft^1_2$. Then $\varrho_s$ just turns $h\varrho$ into an element of the sandwich set $S((h^rh^c)\varrho,(h^ch^l) \varrho)$. The semigroup $\ftt$ is defined as $\ft^1_2\setminus\{1\varrho\}$ ($\{1\varrho\}$ is the group of units of $\ft^1_2$). We can also easily see that $1\varrho$ is constituted by all words of $\Hxp$ with content $\{1\}$.

A nontrivial \emph{i-mountain} (called mountain in \cite{LO22}) is a word $h_0h_1\cdots h_{2n}\in\Hxp$, with $n\geq 1$ and $h_i\in\Hx$ for all $0\leq i\leq 2n$, such that
\begin{itemize}
\item[$(i)$] $h_0=1=h_{2n}$;
\item[$(ii)$] $h_{i-1}\in\{h_i^r,h_i^l\}$ for all $1\leq i\leq n$; and
\item[$(iii)$] $h_i\in\{h_{i-1}^r,h_{i-1}^l\}$ for all $n<i\leq 2n$.
\end{itemize} 
In \cite{LO22} we proved that each $\varrho$-class of $\ftx$ contains a unique nontrivial i-mountain and solved the word problem for $\ftx$ by given a process to construct it. In this paper we will use the term mountain to refer to more complex words.

Let $\mi$ (denoted by $\mx(X)$ in \cite{LO22} for $|X|=2$) denote the set of all nontrivial i-mountains of $\Hxp$. In \cite{LO22} we defined an operation $\odot$ in $\mi$ that turns $\mi$ into a semigroup isomorphic to $\ftt$. The concept of \emph{i-river} (called river in \cite{LO22}) and the process of \emph{uplifting of i-rivers} (called uplifting of rivers in \cite{LO22}) were central. In this paper we will use the terms river and uplifting of rivers again. Their definitions will be very similar to the ones introduced in \cite{LO22}. They differ only on the `context' where they are defined. It will be immediate to see what i-river and uplifting of i-rivers mean from the notions of river and uplifting of rivers introduced here, but the readers may consult \cite{LO22} for further details.

The semigroup $\ftt$ is regular and weakly generated by the two idempotents $e$ and $f$ (identifying each letter with the corresponding $\varrho$-class). It has the following universal property: all regular semigroups weakly generated by two idempotents are homomorphic images of $\ftt$ (under a homomorphism that sends $e\varrho$ and $f\varrho$ into those two idempotents). In Section \ref{sec6} we show that $\ftt$ is embedded into the semigroup $\fo$ that we will construct in the next section.

\section{The presentation $\pres$}\label{sec3}

This section is devoted to the construction of $\fo$ and, after some initial terms are introduced, we will divide it into subsections for a better organization. As we will see, the construction of $\fo$ resembles that of $\ftt$, and it will be easier to define and work with the monoid $\foo$ obtained from $\fo$ by adding a new identity. In contrast to the $\ftt$ case, we will work now with 5-tuples since more information is needed to be included in those tuples. Conjugation will have now an important role too and the two extra entries will contain information about when and where to apply conjugation. 

We begin by defining a presentation $\pres$ for $\foo$. From now on, and if nothing is said in contrary, $X$ will be the singleton set $\{x\}$. Let $A$ be the set $\{1,x,x'\}$. We endow $A$ with a natural involution $'$ by setting $(x')'=x$ and $1'=1$. We call \emph{anchors} to the elements of $A$. As should be expected, $x'$ will represent the inverse of $x$ in $\foo$, while $1$ will correspond to its identity element. For that reason, we will make some abuse of terminology from now on by calling $a'$ the inverse (anchor) of $a$, for any $a\in A$.

The generator set $\G$ will contain the set $A$. All other elements from $\G$ will be special 5-tuples. We will denote the set of all those 5-tuples by $\G^5$. Thus $G=A\cup \G^5$. The set $\G^5$ will be defined recursively as the union of sets $\G_i$ for $i\geq 1$. In fact, each $\G_i$ will be the union of two other sets $\G_{i,e}$ and $\G_{i,d}$. We will describe these sets in more detail below.

We will use the notation 
$$g=(g^l,g^{la},g^c,g^{ra},g^r)$$
to refer to the entries of a 5-tuple $g$ of $\G^5$. The entries $g^{la}$ and $g^{ra}$ will be always anchors. For that reason, they will be called the \emph{left anchor} and the \emph{right anchor} of $g$, respectively. The entries $g^l$, $g^c$ and $g^r$ will be called the \emph{left}, \emph{middle} and \emph{right} entries of $g$. The entries $g^l$ and $g^r$ will be 5-tuples of $\G_{i-1}$ if $g\in\G_i$ for $i>1$; and the entry $g^c$ will be a 5-tuple of $\G_{i-2}$ if $g\in\G_i$ for $i>2$.

As we can see, the left, middle and right entries of a 5-tuple $g\in\G^5$ are usually 5-tuples again. If we need to refer to an entry of $g^l$ for example, we will use the abbreviation $g^{ls}$ to refer to the entry $(g^l)^s$ of $g^l$, for $s\in\{l,la,c,ra,r\}$. But, some entries of $g^l$ may be again 5-tuples. In general, we convene that
$$g^{s_1\cdots s_n}=(((g^{s_1})^{s_2})\cdots)^{s_n}\,,$$
for $s_1,\cdots , s_{n-1}\in\{l,c,r\}$ and $s_n\in\{l,la,c,ra,r\}$, whenever the right side of this equality makes sense. We consider also $g^1=g$.

To help following the remaining of this section, we will divide it into subsections. In the first two subsections we define the sets $\G_{i,e}$ and $\G_{i,d}$, respectively. It is important that the reader pays some attention to the subtleties involved in these definitions as they may be important to understand some of the arguments presented during the paper. In the following subsection we introduce the set of relations $\R$. This set will give us an interpretation for the 5-tuples of $\G^5$. In the forth subsection some special words are described. Finally, the operation of uplifting of rivers is introduced in the fifth subsection and used in the last subsection to present a solution to the word problem for this presentation.

\subsection{The sets $\G_{i,e}$}

We begin by setting $g_{xx'}$ as the 5-tuple
$$g_{xx'}=(1,1,x,x',1)\,.$$
The 5-tuple $g_{xx'}$ will represent the element $xx'$ of $\fo$. For technical reasons that will become clear later, it is convenient to consider the 5-tuple $g_{xx'}$ instead of $xx'$. Set $\G_{1,e}=\{g_{xx'}\}$ and define 
$$\G_{i,e}=\big\{\big(g,(g^{la})',g^l,(g^{ra})',g\big),\,\big(g,(g^{ra})',g^l,(g^{la})',g\big)\,:\;g\in\G_{i-1,e}\big\}$$
for $i\geq 2$. To help clarifying which tuples belong to these sets, note that $\G_{2,e}$ has two elements:\label{G2e}
$$g_{2,e,1}=(g_{xx'},1,1,x,g_{xx'})\quad\mbox{ and }\quad g_{2,e,2}=(g_{xx'},x,1,1,g_{xx'})\;;$$
while $\G_{3,e}$ has four elements:
$$\begin{array}{l}
	(g_{2,e,1},1,g_{xx'},x',g_{2,e,1})\,,\quad (g_{2,e,1},x',g_{xx'},1,g_{2,e,1})\,,\\ [.2cm]
	(g_{2,e,2},1,g_{xx'},x',g_{2,e,2})\quad\mbox{ and }\quad(g_{2,e,2},x',g_{xx'},1,g_{2,e,2})\,.
\end{array}$$

Analyzing the structure of the 5-tuples of $\G_{i,e}$ in more detail for $i>1$, we see that $g\in\G_{i,e}$ if and only if
\begin{itemize}
\item[$(i)$] $g^l=g^r\in\G_{i-1,e}$, 
\item[$(ii)$] $g^c=g^{l^2}=g^{lr}$, and 
\item[$(iii)$] either $\{g^{la},g^{ra}\}=\{1,x\}$ if $i$ even or $\{g^{la},g^{ra}\}=\{1,x'\}$ if $i$ odd.
\end{itemize} 
Hence, the set $\G_{i,e}$ has twice the number of elements of $\G_{i-1,e}$, and so $\G_{i,e}$ has $2^{i-1}$ elements. Further,
$$\G_{i,e}\subseteq \G_{i-1,e}\times A\times\G_{i-2,e}\times A\times\G_{i-1,e}\,,$$
for $i>1$, if one considers $\G_{0,e}=\{1\}$. Let $\G_e=\cup_{i\in\mathbb{N}}\G_{i,e}$. The 5-tuples of $\G_e$ will correspond to the 5-tuples of $\G$ with the same left and right entry.

By construction, both $g^{l^2}$ and $g^{lr}$ are equal to $g^c$ for each $g\in\G_{i,e}$ with $i\geq 2$. Of course, they have different meanings: $g^c$ represents the middle entry of $g$, while $g^{l^2}$ and $g^{lr}$ represent the left and right entries of $g^l$, respectively. In the special words that we will need to consider later, the 5-tuples of $\G$ will have to be `anchored' to each other. For that, we need to attribute a side to $g^c$ inside $g^l$ in order to identify `the anchor of $g^c$ in $g^l$ with respect to $g$'. Note that only one of the elements $g^{l^2a}$ and $g^{lra}$ is equal to  $(g^{la})'$. For each $g$, we define
$$l_a=\left\{\begin{array}{ll}
l^2 & \mbox{ if } g^{l^2a}=(g^{la})' \\ [.2cm]
lr & \mbox{ if } g^{lra}=(g^{la})'\;.
\end{array}\right.$$
Therefore, the anchors $g^{l_aa}$ and $(g^{la})'$ are equal. 

We reinforce that the purpose of the notation $l_a$ is to identify which anchor $g^{l^2a}$ or $g^{lra}$ is equal to $(g^{la})'$, and that $l_a$ differs according to $g$. Hence, the use of $l_a$ only makes sense when we refer to a specific 5-tuple $g\in\G_i$ with $i\geq 2$. Note also that $g^{l_a}$ and $g^c$ have the same `value'. However, when we write $g^{l_a}$, we want to emphasize not only its `value', but also a specific entry of the 5-tuple $g^l$, either its left or its right entry.

In a similar manner we will use also the notation $r_a$. For each $g\in\G_{i,e}$ with $i\geq 2$,
$$r_a=\left\{\begin{array}{ll}
	r^2 & \mbox{ if } g^{r^2a}=(g^{ra})' \\ [.2cm]
	rl & \mbox{ if } g^{rla}=(g^{ra})'\;.
\end{array}\right.$$
We alert again that the use of $r_a$ only makes sense when associated with a specific 5-tuple $g\in\G_{i,e}$ with $i\geq 2$.

\subsection{The sets $\G_{i,d}$}

We now introduce the 5-tuples of $\G$ with distinct left and right entries. They will be described recursively in the sets $\G_{i,d}$ for $i\geq 1$. In contrast with the previous case, the left and right entries of a 5-tuple $g$ from $\G_{i,d}$ may belong not only to $\G_{i-1,d}$, but also to $\G_{i-1,e}$. Thus, we set $\G_i=\G_{i,e}\cup\G_{i,d}$. As we will see next, the definition of the sets $\G_{i,d}$ is more intricate than the definition of the sets $\G_{i,e}$.

In fact, the definition of the sets $\G_{i,d}$ is only interesting for $i>2$, since we set $\G_{1,d}$ and $\G_{2,d}$ as empty sets. Basically, what it says is that for $i=1$ and $i=2$ there are no 5-tuples in $\G$ with distinct left and right entries. Now, assume that $i>2$ and that $\G_{j,d}$ is defined for $j<i$. Consequently, the sets $\G_j$ are also defined. Let $\G_{i,d}$ be the set of all 5-tuples
$$g\in\overline{\G}_i=\G_{i-1}\times A\times\G_{i-2}\times A\times\G_{i-1}$$
such that
\begin{itemize}
\item[$(i)$] $g^l\neq g^r$;
\item[$(ii)$] $\big(g^c,g^{la}\big)=\big(g^{l^2},(g^{l^2a})'\big) $ or $\big(g^c,g^{la}\big)=\big(g^{lr},(g^{lra})'\big)$; 
\item[$(iii)$] $\big(g^c,g^{ra}\big)=\big(g^{rl},(g^{rla})'\big)$ or 
$\big(g^c,g^{ra}\big)=\big(g^{r^2},(g^{r^2a})'\big)$. 
\end{itemize}
From $(i)$ we know that the left entry of $g$ is distinct from its right entry. Note that $(ii)$ tells us, in particular, that $g^c$ is the left or the right entry of $g^l$. In fact, it tells us something stronger: either $g^c$ is the left entry of $g^l$ and the left anchor of $g$ is the inverse of the left anchor of $g^l$, or $g^c$ is the right entry of $g^l$ and the left anchor of $g$ is the inverse of the right anchor of $g^l$. Condition $(iii)$ is just the dual of $(ii)$ with respect to $g^r$.

Next, we list the elements of $\G_{3,d}$ as an example to help understand these sets. It has already 8 elements:
$$\begin{array}{ll}\label{G3d}
g_{3,d,1}=(g_{2,e,1},1,g_{xx'},1,g_{2,e,2})\;,\qquad &(g_{2,e,1},1,g_{xx'},x',g_{2,e,2})\;,\\ [.2cm]
g_{3,d,2}=(g_{2,e,1},x',g_{xx'},x',g_{2,e,2})\;, \qquad &(g_{2,e,1},x',g_{xx'},1,g_{2,e,2})\;, \\ [.2cm]
g_{3,d,3}=(g_{2,e,2},1,g_{xx'},1,g_{2,e,1})\;,\qquad &(g_{2,e,2},1,g_{xx'},x',g_{2,e,1})\;, \\ [.2cm]
g_{3,d,4}=(g_{2,e,2},x',g_{xx'},x',g_{2,e,1})\;,\qquad & (g_{2,e,2},x',g_{xx'},1,g_{2,e,1})\;. 
\end{array}$$
As the reader may have notice, we have attributed a `name' only to the elements of the first column. The reason is because only these elements will be important for Section \ref{sec6}. The elements from the second column will not be used there. It is also a combinatorial exercise to compute the cardinality of $\G_{i,d}$ and $\G_i$. Since we do not need that information for this paper, we leave that computation for the reader. We just point out the obvious fact that the size of the sets $\G_{i,d}$, $\G_{i,e}$ and $\G_i$ increases exponentially with $i$.

Like in the previous subsection, for each $g\in \G_{i,d}$, we need to attribute a side to $g^c$ inside $g^l$. In this case, there are two distinct situations accordingly to $g^l\in\G_{i-1,e}$ or $g^l\in\G_{i-1,d}$. If $g^l\in\G_{i-1,e}$, then $g^{l^2}$ and $g^{lr}$ are both equal to $g^c$. In this case, due to condition $(ii)$ above, we can define $l_a$ in the same manner as we did in the previous subsection. If $g^l\in\G_{i-1,d}$, then the definition of $l_a$ is more natural since only one of left and right entries of $g^l$ is equal to $g^c$: let $l_a=l^2$ if $g^{l^2}=g^c$ or let $l_a=lr$ if $g^{lr}=g^c$. The notation $r_a$ is now introduced similarly (see the previous subsection). We alert once more that these two notations only make sense when referring to a specific 5-tuple $g$.

By definition of both $\G_{i,d}$ and $\G_{i,e}$, observe that
$$g^{l_aa}=(g^{la})'\quad\mbox{ and }\quad g^{r_aa}=(g^{ra})'\,$$
in all cases where the notation makes sense (that is, except for $g=g_{xx'}$). This observation will be used frequently in this paper without further notice. A quick inspection to the definition of $g^{la}$ and $g^{ra}$ allow us to conclude also that $g^{la},g^{ra}\in\{1,x\}$ if $\up(g)$ even, while $g^{la},g^{ra}\in\{1,x'\}$ if $\up(g)$ odd. In fact, if $g\in\G_e$, we can say more: $\{g^{la},g^{ra}\}=\{1,x\}$ if $\up(g)$ even, while $\{g^{la},g^{ra}\}=\{1,x'\}$ if $\up(g)$ odd. If $g\in\G_{i,d}$, we may have $g^{la}=g^{ra}$ as can be observed in some elements of $\G_{3,d}$.

Note that the set $\G_{i,e}$ is also contained in $\overline{\G}_i\,$; whence $\G_i \subseteq \overline{\G}_i$. Let $\G_d=\cup_{i\in\mathbb{N}}\G_{i,d}$ and  $\G^5=\cup_{i\in\mathbb{N}}\G_i=\G_d\cup\G_e$. Set $\G=\G^5\cup A$ and $\G'=\G^5\cup\{1\}$. Thus $\G^5$ is the set of all 5-tuples of $\G$, and it is divided into the set $\G_e$ of all 5-tuples of $\G$ with the same left and right entries, and the set $\G_d$ of all 5-tuples of $\G$ with distinct left and right entries. The \emph{height} $\up(g)$ of an element $g\in\G'$ is defined as follows:
$$\up(1)=0\quad\mbox{ and }\quad \up(g)=i\,\mbox{ if }\, g\in\G_i\,.$$

\subsection{The relation $\R$}

Let us look now to the relation $\R$. We will denote by $\foo$ the semigroup given by the presentation $\pres$, that is, $\foo=\Gp/\rho$ where $\rho$ is the smallest congruence containing $\R$. We will write $u\approx v$ to indicate that the words $u$ and $v$ of $\Gp$ are $\rho$-equivalent. We will use also the notation $[u]$ to refer to the $\rho$-class $u\rho$. We want: 
\begin{itemize}
\item[$(i)$] $[x']$ to be an inverse of $[x]$ such that $[xx']=[g_{xx'}]$; 
\item[$(ii)$] $[1]$ to be the identity element of $\foo$; and
\item[$(iii)$] $[g]$ to be an idempotent for all $g\in\G^5$. 
\end{itemize}
Thus, we include in $\R$ the relation
$$\rho_e=\{(xx'x,x),(x'xx',x'),(g_{xx'},xx'),(1g,g),(g1,g),(g^2,g):g\in\G'\}\,.$$

We immediately see now that we could have omitted $g_{xx'}$ from $\G$. However, its inclusion gives a better feeling about the construction of $\G^5$: this list derives form the `idempotent' $xx'$, and not from $x$ or $x'$. Furthermore, we will see that the congruence classes of the 5-tuples of $\G^5$ constitute a transversal set for the $\Dc$-classes of $\fo$. We should also point out that we could have replaced $xx'$ with $x'x$ in the theory developed in this paper (we just needed to adapt the definition of $\G^5$ for the $x'x$ case). But the definition of $\R$ is not yet finished. We still need to include the conditions that relate $g$ with its entries. For that, we need to explain carefully some notation we will use in order to avoid future ambiguities. 

Given three letters $g_1,g\in\G'$ and $a\in A$, the triplet $g_1ag$ is \emph{left anchored} if $$(g_1,a)=(g^l,g^{la})\quad\mbox{ or }\quad (g_1,a)=(g^r,g^{ra})\,.$$ 
The notion of right anchored is not quite the dual: the triplet $gag_1$ is \emph{right anchored} if
$$(g_1,a)=(g^l,(g^{la})')\quad\mbox{ or }\quad (g_1,a)=(g^r,(g^{ra})')\,.$$
Observe that in the definition of left anchored, the anchor $a$ is one of the anchors of $g$, while in the definition of right anchored, the anchor $a$ is the inverse of one of the anchors of $g$. In both cases, we say that $g_1$ is \emph{anchored} to $g$. Finally, a triplet $g_1ag_2$, with $g_1,g_2\in\G'$ and $a\in A$, is \emph{anchored} if $g_1$ is anchored to $g_2$ or $g_2$ is anchored to $g_1$.

For $s\in\{l,r\}$, we may denote the anchored triplets $g^sg^{sa}g$ and $g(g^{sa})'g^s$ by $g^s\cdot g$ and $g\cdot g^s$, respectively. Note that this notation may become ambiguous if we don't indicate explicitly the value of $s$. For example, if $g\in\G_e$ and $h=g^l$, then the expression $h\cdot g$ is ambiguous. Note that $h$ is also equal to $g^r$ and, therefore, $h\cdot g$ may refer to either $g^lg^{la}g$ or $g^rg^{ra}g$, with $g^{la}\neq g^{ra}$. Hence, we must be careful when using this `dot' notation.

There is another instance where this `dot' notation may be useful but ambiguous too if we don't clarify it, namely $g^c\cdot g^s$ and $g^s\cdot g^c$ for $s\in\{l,r\}$. Note that $g^c$ may be equal to both $g^{sl}$ and $g^{sr}$. We define 
$$g^c\cdot g^s=g^cg^{s_aa}g^s=g^c(g^{sa})'g^s\quad\mbox{ and }\quad g^s\cdot g^c=g^s(g^{s_aa})'g^c=g^sg^{sa}g^c\,.$$
Clearly, these triplets are anchored by the way they are defined.

Next, we introduce the (non-anchored) triplets $g^L$ and $g^R$. Let $g^L=(g^{la})'g^lg^{la}$ and $g^R=(g^{ra})'g^rg^{ra}$. Note that
$$g^L\in\{1g^l1,\,xg^lx',\,x'g^lx\}\,.$$
In fact, we can say also that $g^L\neq xg^lx'$ if $\up(g)$ even, and $g^L\neq x'g^lx$ if $\up(g)$ odd. Similar conclusions are also true for $g^R$. Note now that
$$g\!\cdot\! g^l\!\!\cdot\! g = gg^Lg\,,\;\; g^c\!\!\cdot\! g^l\!\!\cdot\! g=g^cg^Lg\,,\;\; g\!\cdot\! g^r\!\!\cdot\! g=gg^Rg \;\mbox{ and }\; g\!\cdot\! g^r\!\!\cdot\! g^c=gg^Rg^c.$$
The next observation is also useful. Since $\;g^r\!\cdot g\cdot g^l=g^rg^{ra}g(g^{la})'g^l\;$ and 
$$g^r\!\cdot g^c\cdot g^l=g^r(g^{r_aa})'g^cg^{l_aa}g^l=g^rg^{ra}g^c(g^{la})'g^l\,,$$
the words $g^r\!\cdot g\cdot g^l$ and $g^r\!\cdot g^c\cdot g^l$ of length 5 differ only on the middle letter: we have $g$ in $g^r\!\cdot g\cdot g^l$ and $g^c$ in $g^r\!\cdot g^c\cdot g^l$.

Before we continue, this is a good place to make a pause and explain the main difference between the theory developed here and the one developed in \cite{LO22}. This explanation may help also the readers familiar with \cite{LO22} following this paper. We have already mentioned the close resemblance between the structure of these two papers. The main difference is on what we will consider to be landscapes and mountains. In \cite{LO22}, they were words from $\Hxp$ where any two consecutive letters had a special connection (see Section \ref{sec2}). Here, they will be words constituted by alternating letters from $\G'$ and $A$ such that any triplet subword with two letters from $\G'$ and one from $A$ is anchored. The notation introduced above will help us, in some cases, to give a less cumbersome aspect to those words.

Let $\rho_s$ be the following relation on $\G^+$:
$$\rho_s=\left\{(g^cg^Lg,g),(gg^Rg^c,g),(g^r\!\!\cdot\! g\!\cdot\! g^l,g^r\!\!\cdot\! g^c\!\cdot\! g^l)|\, g\in \G \mbox{ with }\up(g)\geq 2\right\}.$$
We set $\R=\rho_e\cup\rho_s$. Observe that $[1]$ is composed by all words from $\Gp$ of content $\{1\}$. Thus $\foo\setminus\{[1]\}$ is a subsemigroup of $\foo$. Let $\fo=\foo\setminus\{[1]\}$.

The next result gives a list of properties about the product in $\foo$ and terminates with an interpretation for $\rho_s$. We already know that $[g]$ is an idempotent of $\foo$ due to the definition of $\rho_e$. We will prove that $[g^cg^R]$ and $[g^Lg^c]$ are also idempotents of $\foo$. Then, the definition of $\rho_s$ is just to assure that $[g]\in S([g^Rg^c],[g^cg^L])$.

\begin{lem}\label{prodelem}
Let $g\in \G$ with $\up(g)>1$. Then:
\begin{itemize}
\item[$(i)$] $g\ap g^cg\ap gg^c\ap gg^Lg\ap gg^Rg$.
\item[$(ii)$] $[gg^L],\,[gg^R],\,[g^Lg],\,[g^Rg]\in E(\foo)$.
\item[$(iii)$] $[g^l\!\!\cdot\! g],[g^Lg],[g^r\!\!\cdot\! g],[g^Rg]\in \Lcc_{[g]}$ and $[g\!\cdot\! g^l],[gg^L],[g\!\cdot\! g^r],[gg^R]\in \Rcc_{[g]}$.
\item[$(iv)$] $g^Lg^cg^L\ap g^L$ and $g^Rg^cg^R\ap g^R$. 
\item[$(v)$] $[g^cg^L],\,[g^cg^R],\,[g^Lg^c],\,[g^Rg^c]\in E(\foo)$.
\item[$(vi)$] $[g]\in S([g^Rg^c],[g^cg^L])$.
\end{itemize}
\end{lem}

\begin{proof}
Note that $gg^c\ap gg^Rg^cg^c\ap gg^Rg^c\ap g$ and 
$gg^Lg\ap gg^cg^Lg\ap g^2\ap g$ 
by definition of $\rho_e$ and $\rho_s$. The proof of $(i)$ is now complete by duality. Then $(ii)$ follow obviously from $(i)$. Further, $[g^Lg] $ and $[g^Rg]$ also belong to $\Lcc_{[g]}$. Then, since $(g^{la})'g^l\cdot g=g^Lg$ and $(g^{ra})'g^r\cdot g=g^Rg$, we must have both $[g^l\cdot g]$ and $[g^r\cdot g]$ in $\Lcc_{[g]}$ too. Thus $(iii)$ holds by duality.

Note now that
$$\begin{array}{ll}
	g^Lg^cg^L\hspace*{-.2cm}&=(g^{la})'g^lg^{la} g^c(g^{la})' g^lg^{la}=(g^{la})'g^l(g^{l_aa})'g^cg^{l_aa}g^lg^{la}=\\ [.3cm]
	&=\left\{\begin{array}{ll}
		(g^{la})'g^l(g^l)^Lg^lg^{la} &\mbox{ if } l_a=l^2 \\ [.2cm]
		(g^{la})'g^l(g^l)^Rg^lg^{la} &\mbox{ if } l_a=lr \\ [.2cm]
	\end{array}\right.
	\end{array}$$
Applying statement $(i)$ to $g^l$ we get $g^Lg^cg^L\ap (g^{la})'g^lg^{la}=g^L$. Similarly, we show that $g^Rg^cg^R\ap g^R$ and $(iv)$ is proved. Then $(v)$ is an obvious consequence of $(iv)$.
 
Finally, to prove $(vi)$, we already know that $[g]$, $[g^Rg^c]$ and $[g^cg^L]$ are idempotents of $\foo$ such that $g^cg^Lg\ap g$ and $gg^Rg^c\ap g$. So, we only need to prove that $g^Rg^cgg^cg^L\ap g^Rg^cg^L$. But
$$g^Rg^cgg^cg^L\ap g^Rgg^L= (g^{ra})'g^r\cdot g\cdot g^lg^{la}\ap (g^{ra})'g^r\cdot g^c\cdot g^lg^{la} = g^Rg^cg^L\,.$$
Hence, we proved that $[g]\in S([g^Rg^c],[g^cg^L])$.
\end{proof}

\subsection{Special words}
 
To facilitate the organization of this paper, if nothing is said in contrary, we will use $g$ and $h$ (with possible indices) to refer to elements of $\G'$, $a$ and $b$ (with possible indices) to refer to anchors, and $u$ and $v$ (also with possible indices) to refer to words from $\Gp$. Next, we list some general terms and notations that we will use throughout this paper. These terms identify special words from $\Gp$ and their characteristics that will be important to us. We use the same terminology of \cite{LO22} and the notions are very much alike as the reader can verify. The main difference, as mentioned already, is on what we consider now a landscape.

\begin{description}
\item[Landscape] alternating word $u=g_0a_1g_1\cdots a_ng_n\in \Gp$ ($n\geq 0$) of letters $g_j\in\G'$ and anchors $a_j\in A$ such that all triplet subwords $g_{i-1}a_ig_i$ are anchored. We denote by $\lx$ the set of all landscapes of $\Gp$. Note also that single letters $g\in \G'$ are particular landscapes; whence $\G'\subseteq \lx$. We make notice that some of the terms below refer to properties only of the letters $g_i$ of a landscape (and not to its anchors $a_i$).
\item[Ridge] letter $g_i$ of a landscape $u=g_0a_1g_1\cdots a_ng_n$ such that $\up(g_{i-1})=\up(g_{i+1})=\up(g_i)-1$. Hence, a ridge is never the first nor the last letter of a landscape. 
\item[Peak] highest ridge of a landscape $u$. Note that $u$ can have several peaks, but all of them have the same height. If $u$ has only one peak, we denote it by $\ka(u)$. The height $\up(u)$ of the landscape $u$ is the maximum between the height of its peaks, the height of $\si(u)$, and the height of $\tau(u)$. Note that this definition of height agrees with the notion of height for elements of $\G$ introduced earlier.
\item[River] letter $g_i$ of a landscape $u=g_0a_1g_1\cdots a_ng_n$ such that $\up(g_{i-1})=\up(g_{i+1})=\up(g_i)+1$. As with ridges, rivers are never endpoints of landscapes.
\item[Hill] landscape $u=g_0a_1g_1\cdots a_ng_n$ with $n\geq 1$ such that either $g_{i-1}\in\{g_i^l,g_i^r\}$ for all $1\leq i\leq n$, or $g_i\in\{g_{i-1}^l,g_{i-1}^r\}$ for all $1\leq i\leq n$. In the former case we have an \emph{uphill} since $\up(g_i)=\up(g_{i-1})+1$, while on the latter case we have a \emph{downhill} since $\up(g_{i})=\up(g_{i-1})-1$. We denote by $\lo$, $\lop$ and $\lom$ the set of all hills, uphills and downhills, respectively. Further, $\ltp$ will denote the set of all landscapes composed by an uphill followed by a downhill.
\item[Valley] landscape composed by a downhill followed by an uphill. Thus, a valley has always one (and only one) river at its lower height letter. We denote by $\ltm$ the set of all valleys of $\Gp$, and by $\lt$ the set $\ltp\cup\ltm$. 
\item[Mountain range] landscape $u$ with $\si(u)=\tau(u)=1$. Thus, the height of a nontrivial mountain range is the height of its peaks. Note that all mountain ranges have length $4k+1$ for some $k\in\mathbb{N}_0$, and $u=1$ is the only mountain range of length $1$. If $u\neq 1$, we denote by $\la_l(u)$ the maximal prefix of $u$ which is also an uphill, and by $\la_r(u)$ the maximal suffix of $u$ which is also a downhill. We call $\la_l(u)$ and $\la_r(u)$ the \emph{left hill} and the \emph{right hill} of $u$, respectively. For technical reasons, we define also $\la_l(1)=\la_r(1)=1$. We denote by $\mrx$ the set of all mountain ranges of $\Gp$.
\item[Mountain] mountain range $u$ with no rivers. Thus, it is either the trivial mountain range $u=1$, or it is composed by the left hill $\la_l(u)$ followed by the right hill $\la_r(u)$. Nontrivial mountains have only one ridge. We denote by $\mxm$ and $\mx$ the sets of all mountains and nontrivial mountains, respectively. So $\mx=\mrx\cap \ltp$.
\end{description}

We shall use the previous terminology also to refer to subwords of a given word $u\in \Gp$. For example, a hill of $u\in\lx$ is a subword of $u$ that is also a hill. A hill of $u\in\lx$ is called maximal if it is not properly contained in another hill of $u$, while a valley of $u$ is called maximal if it is not properly contained in another valley of $u$. Thus, each nontrivial mountain range is uniquely decomposable into its left hill, followed by a possible empty sequence of maximal valleys between its consecutive ridges, and ending with its right hill. However, we must be careful since $u_1u_2$ is not a landscape if $u_1$ and $u_2$ are two landscapes such that $g=\tau(u_1)=\si(u_2)$: it appears a double $gg$ at the junction of $u_1$ with $u_2$. Therefore, we will use the notation $u_1*u_2$ to represent the landscape obtained from $u_1u_2$ by replacing $gg$ with $g$. Note that $u_1u_2\ap u_1*u_2$ since $[g]$ is an idempotent of $\foo$.

We will represent landscapes as line graphs where the anchors label the edges and the other letters label the vertices. We will include information about the height of the letters from $\G'$ in this line graphs as follows: we draw them such that all vertices labeled by letters with the same height belong to the same imaginary horizontal line; and labels of vertices in higher imaginary horizontal lines have higher height. We try to illustrate this idea in Figure \ref{fig1}. For example, looking carefully to that picture, we can conclude that $(g_1,a_2)=(g_2^s,g_2^{sa})$ and $(a_7,g_7)=((g_6^{ta})',g_6^t)$ for some $s,t\in\{l,r\}$. Using these line graphs, we can see that the terminology introduced above has a natural interpretation. We draw again the line graph of Figure \ref{fig1} in Figure \ref{fig2} (omitting some unnecessary information) to illustrate some of those terms.

\begin{figure}[ht]
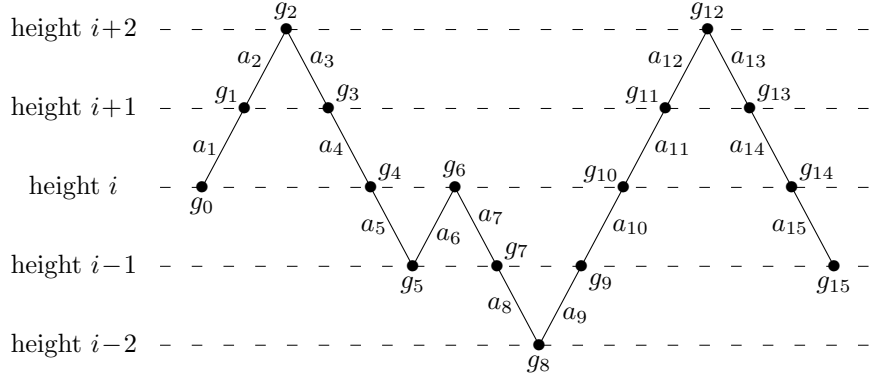
 
$$\tikz[scale=.7]{
\coordinate (1) at (1.8,3);
\coordinate (2) at (2.6,4.5);
\coordinate (3) at (3.4,6);
\coordinate (4) at (4.2,4.5);
\coordinate (5) at (5,3);
\coordinate (6) at (5.8,1.5);
\coordinate (7) at (6.6,3);
\coordinate (8) at (7.4,1.5);
\coordinate (9) at (8.2,0);
\coordinate (10) at (9,1.5);
\coordinate (11) at (9.8,3);
\coordinate (12) at (10.6,4.5);
\coordinate (13) at (11.4,6);
\coordinate (14) at (12.2,4.5);
\coordinate (15) at (13,3);
\coordinate (16) at (13.8,1.5);
\coordinate (17) at (1,0);
\coordinate (18) at (1,1.5);
\coordinate (19) at (1,3);
\coordinate (20) at (1,4.5);
\coordinate (21) at (1,6);
\coordinate (22) at (14.5,0);
\coordinate (23) at (14.5,1.5);
\coordinate (24) at (14.5,3);
\coordinate (25) at (14.5,4.5);
\coordinate (26) at (14.5,6);
\draw (1) node {$\bullet$} node [below] {\small $g_0$};
\draw (2) node {$\bullet$} node [above left=-2pt] {\small $g_1$};
\draw (3) node {$\bullet$} node [above] {\small $g_2$};
\draw (4) node {$\bullet$} node [above right=-1pt] {\small $g_3$};
\draw (5) node {$\bullet$} node [above right=-1pt] {\small $g_4$};
\draw (6) node {$\bullet$} node [below] {\small $g_5$};
\draw (7) node {$\bullet$} node [above] {\small $g_6$};
\draw (8) node {$\bullet$} node [above right=-1pt] {\small $g_7$};		
\draw (9) node {$\bullet$} node [below] {\small $g_8$};		
\draw (10) node {$\bullet$} node [below right=-1pt] {\small $g_9$};
\draw (11) node {$\bullet$} node [above left=-2pt] {\small $g_{10}$};
\draw (12) node {$\bullet$} node [above left=-2pt] {\small $g_{11}$};
\draw (13) node {$\bullet$} node [above] {\small $g_{12}$};	
\draw (14) node {$\bullet$} node [above right=-1pt] {\small $g_{13}$};
\draw (15) node {$\bullet$} node [above right=-1pt] {\small $g_{14}$};		
\draw (16) node {$\bullet$} node [below] {\small $g_{15}$};						
\draw (17) node [left=5pt] {\small height $i\!-\!2$};
\draw (18) node [left=5pt] {\small height $i\!-\!1$};
\draw (19) node [left=12pt] {\small height $i$};
\draw (20) node [left=5pt] {\small height $i\!+\!1$};
\draw (21) node [left=5pt] {\small height $i\!+\!2$};		
\draw[very thin, dashed, dash pattern=on 4pt off 8pt] (17)--(22);
\draw[very thin, dashed, dash pattern=on 4pt off 8pt] (18)--(23);
\draw[very thin, dashed, dash pattern=on 4pt off 8pt] (19)--(24);
\draw[very thin, dashed, dash pattern=on 4pt off 8pt] (20)--(25);
\draw[very thin, dashed, dash pattern=on 4pt off 8pt] (21)--(26);
\draw (1) to node[left=-2pt] {\small$a_1$} (2);
\draw (2) to node[above left=-3pt] {\small$a_2$} (3);
\draw (3) to node[above right=-3pt] {\small$a_3$} (4);
\draw (4) to node[left=-2pt] {\small$a_4$} (5);
\draw (5) to node[left=-2pt] {\small$a_5$} (6);
\draw (6) to node[below right=-3pt] {\small$a_6$} (7);
\draw (7) to node[above right=-3pt] {\small$a_7$} (8);
\draw (8) to node[left=-2pt] {\small$a_8$} (9);
\draw (9) to node[below right=-3pt] {\small$a_9$} (10);
\draw (10) to node[right] {\small$a_{10}$} (11);
\draw (11) to node[right] {\small$a_{11}$} (12);
\draw (12) to node[above left=-3pt] {\small$a_{12}$} (13);
\draw (13) to node[above right=-3pt] {\small$a_{13}$} (14);
\draw (14) to node[left=-2pt] {\small$a_{14}$} (15);	
\draw (15) to node[left=-2pt] {\small$a_{15}$} (16);				
}$$
\caption{Illustration of a landscape in a line graph.}\label{fig1}
\end{figure}

\begin{figure}[ht]
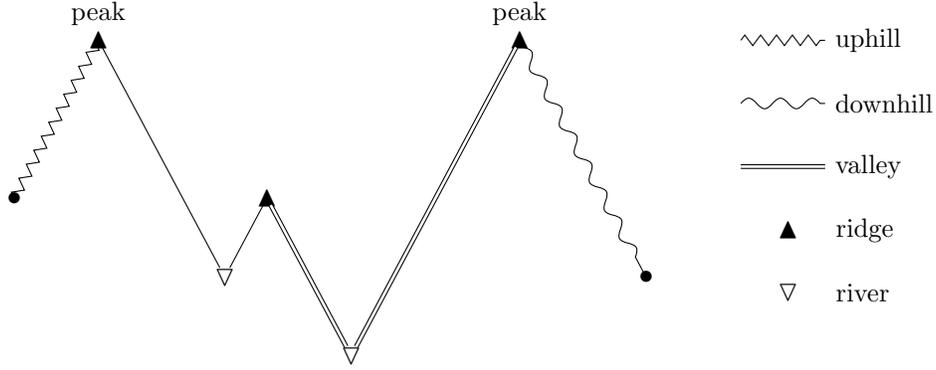
 
$$\tikz[scale=.7]{
\coordinate (1) at (1.8,3);
\coordinate (2) at (3.4,6);
\coordinate (3) at (5.8,1.5);
\coordinate (4) at (6.6,3);
\coordinate (5) at (8.2,0);
\coordinate (6) at (11.4,6);
\coordinate (7) at (13.8,1.5);
\draw (1) node {$\bullet$};
\draw (2) node {$\blacktriangle$} node[above=2pt] {\small peak};
\draw (3) node {$\triangledown$};
\draw (4) node {$\blacktriangle$};
\draw (5) node {$\triangledown$};	
\draw (6) node {$\blacktriangle$}	node[above=2pt] {\small peak};
\draw (7) node {$\bullet$};	
\draw[snake,segment amplitude=2pt,segment length=6pt,gap after snake=1pt] (1)--(2);
\draw[shorten >=4pt, shorten <=2pt] (2)--(3);
\draw[shorten >=2pt, shorten <=4pt] (3)--(4);
\draw[double,double distance=1pt,shorten >=4pt, shorten <=3pt] (4)--(5);
\draw[double,double distance=1pt,shorten >=3pt, shorten <=4pt] (5)--(6);
\draw[snake=coil,segment aspect=0,segment amplitude=2pt,segment length=12pt,gap before snake=3pt] (6)--(7);			
\draw[snake,segment amplitude=2pt,segment length=6pt] (15.6,6)--(17.2,6);
\draw (17.2,6) node[right] {\small uphill};
\draw[snake=coil,segment aspect=0,segment amplitude=2pt,segment length=12pt] (15.6,4.8)--(17.2,4.8);
\draw (17.2,4.8) node[right] {\small downhill};
\draw[double,double distance=1pt] (15.6,3.6)--(17.2,3.6);
\draw (17.2,3.6) node[right] {\small valley};
\draw (16.5,2.4) node {$\blacktriangle$} node[right=.5cm] {\small ridge};
\draw (16.5,1.2) node {$\triangledown$} node[right=.5cm] {\small river};
}$$
\caption{Illustration of some terminology.}\label{fig2}
\end{figure}

Next, we associate a mountain to each element of $\G$. We begin by observing that $11g_{xx'}x1$, $1x'g_{xx'}11$ and $11g_{xx'}11$ are mountains. It may seem strange the appearance of double ones ($11$). This is so because one of them is consider as a letter from $\G'$ while the other is an anchor. Of course, since we know that $1$ will represent the identity element of $\foo$, we could have replaced these double ones by a single one or even omit them. However, this would mean that we would have to treat this particular letter different from the other letters from $\G'$. For technical reason it is better not to do so, and consider the letter $1$ from $\G'$ as any other letter. Therefore, we will see very often these double ones (or even longer sequences of ones) in our landscapes.

We define $\be_1(x)=11g_{xx'}x1$, $\be_1(x')=1x'g_{xx'}11$ and $\be_1(g_{xx'})=11g_{xx'}11$ since $x\ap 11g_{xx'}x1$, $x'\ap 1x'g_{xx'}11$ and $g_{xx'}\ap 11g_{xx'}11$. We define also $\be_1(1)=1$, the trivial mountain. The definition of $\be_1(g)$ for a 5-tuple $g$ of height greater than 1 is more complex. These words are long and, therefore, we will define first their left and right hills.

Let $g\in\G^5$ such that $\up(g)=2n$ or $\up(g)=2n+1$ for some $n\in\mathbb{N}$. Set
$$\be_{1,l}(g)=\underbracket{g^{c^n}g^{c^{n-1}l_aa}g^{c^{n-1}l}g^{c^{n-1}la}}g^{c^{n-1}}\!\cdots \underbracket{g^{c^2}g^{cl_aa}g^{cl}g^{cla}}\underbracket{g^cg^{l_aa}g^lg^{la}}g$$
and 
$$\be_{1,r}(g)=g\underbracket{(g^{ra})'g^r(g^{r_aa})'g^c}(g^{cra})'g^{cr}\!\cdots g^{c^{n-1}}\underbracket{(g^{c^{n-1}ra})'g^{c^{n-1}r}(g^{c^{n-1}r_aa})'g^{c^n}}\!.$$
The `brackets' under these words highlight the four element blocks used as patterns in their construction. To build $\be_{1,l}(g)$, we add to the left of each $g^{c^i}$ the block $g^{c^{i+1}}g^{c^il_aa}g^{c^il}g^{c^ila}$; while, to build $\be_{1,r}(g)$, we add to the right of each $g^{c^i}$ the block $(g^{c^ira})'g^{c^ir}(g^{c^ir_aa})'g^{c^{i+1}}$.
Consequently, $\be_{1,l}(g)$ is an uphill from $g^{c^n}$ to $g$, while $\be_{1,r}(g)$ is a downhill from $g$ to $g^{c^n}$. Furthermore, $g^{c^n}=1$ if $\up(g)=2n$, and $g^{c^n}=g_{xx'}$ if $\up(g)=2n+1$. 

We now define
$$\be_1(g)=\left\{\begin{array}{ll}
\be_{1,l}(g)*\be_{1,r}(g) & \mbox{ if } \up(g)=2n \\ [.2cm]
11\be_{1,l}(g)*\be_{1,r}(g)11 & \mbox{ if } \up(g)=2n+1
\end{array}\right.$$
Clearly, $\be_1(g)$ is a mountain if $\up(g)=2n$. If $\up(g)=2n+1$, then $\be_1(g)$ is also a mountain since both $11g_{xx'}$ and $g_{xx'}11$ are hills. We depict the mountain $\be_1(g)$ in Figure \ref{fig3} to help visualizing it. We denote by $\la_l(g)$ and $\la_r(g)$ the left hill and the right hill of $\be_1(g)$, respectively, for each $g\in\G^5$. Hence $\be_1(g)=\la_l(g)*\la_r(g)$ for all $g\in\G^5$. Note also that $\la_l(g)=\be_{1,l}(g)$ and $\la_r(g)=\be_{1,r}(g)$ if $\up(g)$ is even; while $\la_l(g)=11\be_{1,l}(g)$ and $\la_r(g)=\be_{1,r}(g)11$ if $\up(g)$ is odd.

\begin{figure}[ht]
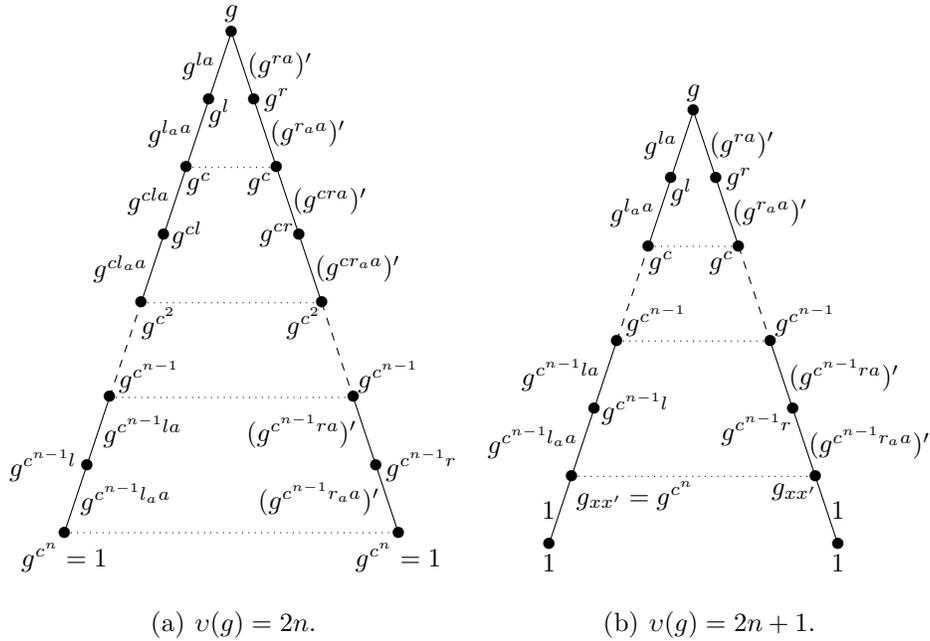

\begin{subfigure}[b]{.49\textwidth}
\centering
$$\tikz[scale=.6]{
\coordinate (1) at (0,0);
\coordinate (2) at (.5,1.5);
\coordinate (3) at (1.2,3.6);
\coordinate (4) at (1.7,5.1);
\coordinate (5) at (2.2,6.6);
\coordinate (6) at (2.7,8.1);
\coordinate (7) at (3.2,9.6);
\coordinate (8) at (3.7,8.1);
\coordinate (9) at (4.2,6.6);
\coordinate (10) at (4.7,5.1);
\coordinate (11) at (5.2,3.6);
\coordinate (12) at (5.9,1.5);
\coordinate (13) at (6.4,0);
\coordinate (14) at (-0.5,-1.5);
\coordinate (15) at (6.9,-1.5);		
\draw (14) node {$\bullet$} node [below] {\small $g^{c^n}=1$};		
\draw (1) node {$\bullet$} node [left] {\small $g^{c^{n-1}l}$};
\draw (2) node {$\bullet$} node [above right=-2pt] {\small $\,g^{c^{n-1}}$};		
\draw (3) node {$\bullet$} node [below right=-2pt] {\small $\!g^{c^2}$};
\draw (4) node {$\bullet$} node [right=-1pt] {\small $g^{cl}$};
\draw (5) node {$\bullet$} node [below right=-2pt] {\small $\!g^c$};		
\draw (6) node {$\bullet$} node [below right=-4pt] {\small $g^l$};		
\draw (7) node {$\bullet$} node [above] {\small $g$};
\draw (8) node {$\bullet$} node [right] {\small $g^r$};		
\draw (9) node {$\bullet$} node [below left=-2pt] {\small $g^c$};
\draw (10) node {$\bullet$} node [left=-2pt] {\small $g^{cr}$};
\draw (11) node {$\bullet$} node [below left=-2pt] {\small $g^{c^2}\!$};		
\draw (12) node {$\bullet$} node [above right=-2pt] {\small $g^{c^{n-1}}$};			
\draw (13) node {$\bullet$} node [right] {\small $g^{c^{n-1}r}$};
\draw (15) node {$\bullet$} node [below] {\small $g^{c^n}=1$};		
\draw (14) to node[right=-2pt] {\small$g^{c^{n-1}l_aa}$} (1);
\draw (1) to node[right=-2pt] {\small$g^{c^{n-1}la}$} (2);		
\draw[dashed] (2)--(3);
\draw (3) to node[left=-2pt] {\small$g^{cl_aa}$} (4);
\draw (4) to node[left=-2pt] {\small$g^{cla}$} (5);
\draw (5) to node[left=-2pt] {\small$g^{l_aa}$} (6);
\draw (6) to node[left=-2pt] {\small$g^{la}$} (7);
\draw (7) to node[right=-2pt] {\small$(g^{ra})'$} (8);
\draw (8) to node[right=-2pt] {\small$(g^{r_aa})'$} (9);
\draw (9) to node[right=-2pt] {\small$(g^{cra})'$} (10);
\draw (10) to node[right=-2pt] {\small$(g^{cr_aa})'$} (11);
\draw[dashed] (11)--(12);
\draw (12) to node[left=-1pt] {\small$(g^{c^{n-1}ra})'$} (13);
\draw (13) to node[left=-1pt] {\small$(g^{c^{n-1}r_aa})'$} (15);				
\draw[dotted] (14)--(15);
\draw[dotted] (3)--(11);
\draw[dotted] (5)--(9);		
\draw[dotted] (2)--(12);				
}$$	
\caption{$\up(g)=2n$.}
\end{subfigure}
\begin{subfigure}[b]{.49\textwidth}
\centering
$$\tikz[scale=.6]{
\coordinate (1) at (0,0);			
\coordinate (2) at (0.5,1.5);
\coordinate (3) at (1,3);
\coordinate (4) at (1.5,4.5);
\coordinate (5) at (2.2,6.6);
\coordinate (6) at (2.7,8.1);
\coordinate (7) at (3.2,9.6);
\coordinate (8) at (3.7,8.1);
\coordinate (9) at (4.2,6.6);
\coordinate (10) at (4.9,4.5);
\coordinate (11) at (5.4,3);
\coordinate (12) at (5.9,1.5);
\coordinate (13) at (6.4,0);		
\draw (1) node {$\bullet$} node [below] {\small $1$};
\draw (2) node {$\bullet$} node [below right=-2pt] {\small $g_{xx'}=g^{c^n}$};		
\draw (3) node {$\bullet$} node [right=-1pt] {\small $g^{c^{n-1}l}$};
\draw (4) node {$\bullet$} node [above right=-2pt] {\small $\,g^{c^{n-1}}$};		
\draw (5) node {$\bullet$} node [below right=-2pt] {\small $\!g^c$};		
\draw (6) node {$\bullet$} node [below right=-4pt] {\small $g^l$};		
\draw (7) node {$\bullet$} node [above] {\small $g$};
\draw (8) node {$\bullet$} node [right] {\small $g^r$};		
\draw (9) node {$\bullet$} node [below left=-2pt] {\small $g^c$};
\draw (10) node {$\bullet$} node [above right=-2pt] {\small $g^{c^{n-1}}$};			
\draw (11) node {$\bullet$} node [below left=-2pt] {\small $g^{c^{n-1}r}\!$};
\draw (12) node {$\bullet$} node [below left=-1pt] {\small $g_{xx'}\!\!$};
\draw (13) node {$\bullet$} node [below] {\small $1$};		
\draw (1) to node[left=-2pt] {\small$1$} (2);
\draw (2) to node[left=-2pt] {\small$g^{c^{n-1}l_aa}$} (3);
\draw (3) to node[left=-2pt] {\small$g^{c^{n-1}la}$} (4);		
\draw[dashed] (4)--(5);
\draw (5) to node[left=-2pt] {\small$g^{l_aa}$} (6);
\draw (6) to node[left=-2pt] {\small$g^{la}$} (7);
\draw (7) to node[right=-2pt] {\small$(g^{ra})'$} (8);
\draw (8) to node[right=-2pt] {\small$(g^{r_aa})'$} (9);
\draw[dashed] (9)--(10);
\draw (10) to node[right=-1pt] {\small$(g^{c^{n-1}ra})'$} (11);
\draw (11) to node[right=-2pt] {\small$(g^{c^{n-1}r_aa})'$} (12);	
\draw (12) to node[right=-2pt] {\small$1$} (13);			
\draw[dotted] (2)--(12);
\draw[dotted] (5)--(9);		
\draw[dotted] (4)--(10);				
}$$	
\caption{$\up(g)=2n+1$.}
\end{subfigure}
\caption{Illustration of the mountain $\be_1(g)$.}\label{fig3}
\end{figure}

Let $u=h_0\cdots h_k$ with $h_i\in\G$ for $0\leq i\leq k$. Extend $\be_1$ to $u$ by setting
$$\be_1(u)=\be_1(h_0)*\cdots*\be_1(h_k)\,.$$ 
Then $\be_1(u)\in \mrx$. In the following result we prove that $u\ap\be_1(u)$ for all words $u\in\Gp$.

\begin{lem}\label{m1}
$\be_1(u)\ap u$ for all $u\in \Gp$.
\end{lem}

\begin{proof}
Note that we just need to prove that $\be_1(h)\ap h$ for all $h\in\G$. Indeed, if $\be_1(h)\ap h$ for all $h\in\G$, then
$$\be_1(u)\ap\be_1(h_0)\be_1(h_1)\cdots\be_1(h_k)\ap h_0h_1\cdots h_k=u\,.$$
We have already observed that $\be_1(a)\ap a$ for all $a\in A$. Clearly 
$$g^{c^i}(g^{c^{i-1}})^{l_aa}g^{c^{i-1}l}g^{c^{i-1}la}g^{c^{i-1}}=(g^{c^{i-1}})^c\, (g^{c^{i-1}})^L\, g^{c^{i-1}}\ap g^{c^{i-1}}$$
by definition of $\rho_s$. Since we have also $11g_{xx'}\ap g_{xx'}$, we conclude that $\la_l(g)\ap\be_{1,l}(g)\ap g$ for all $g\in\G^5$. Dually, we can conclude also that $\la_r(g)\ap\be_{1,r}(g)\ap g$. Consequently
$$\be_1(g)=\la_l(g)*\la_r(g)\ap \la_l(g)\la_r(g)\ap g^2\ap g\,.$$
We have proved that $\be_1(h)\ap h$ for all $h\in\G$ as wanted.
\end{proof}

\subsection{Uplifting of rivers}

Let $u=g_0a_1g_1\cdots a_ng_n$ be a landscape and assume that $g_i$ is a river of $u$. Then $g_{i-1}a_ig_ia_{i+1}g_{i+1}$ is a subword of $u$ such that $g_{i-1}a_ig_i$ is a right anchored triplet and $g_ia_{i+1}g_{i+1}$ is a left anchored triplet. In other words,
$$(g_i,a_i)=\big(g_{i-1}^s,(g_{i-1}^{sa})'\big)\quad\mbox{ and }\quad (g_i,a_{i+1})= (g_{i+1}^t,g_{i+1}^{ta})\,,$$
for some $s,t\in\{l,r\}$. 

Let $h_i$ be the 5-tuple $(g_{i+1},a_{i+1}',g_i,a_i,g_{i-1})$. If $g_{i-1}\neq g_{i+1}$, then $h_i\in\G_d\,$; and if $g_{i-1}=g_{i+1}$ and $a_i\neq a_{i+1}'$, then $h_i\in\G_e$. However, $h_i\not\in\G^5$ if $g_{i-1}=g_{i+1}$ and $a_i=a_{i+1}'$. So, set
$$v=\left\{\begin{array}{ll}
g_0a_1\cdots g_{i-1}a_{i+2}g_{i+2}\cdots g_n &\mbox{ if } g_{i-1}=g_{i+1} \mbox{ and } a_i=a_{i+1}\,,\\ [.2cm]
g_0\cdots g_{i-1}a_ih_ia_{i+1}g_{i+1}\cdots g_n\quad &\mbox{ otherwise}\,. 
\end{array}\right.$$
If $h_i\in\G^5$, then the triplet $g_{i-1}a_ih_i$ is left anchored, while the triplet $h_ia_{i+1}g_{i+1}$ is right anchored; whence $v$ is also landscape. If $h_i\not\in\G^5$, then $g_{i-1}a_{i+2}g_{i+2}=g_{i+1}a_{i+2}g_{i+2}$ is an anchored triplet and again $v$ is a landscape too. We have shown that $v$ is always a landscape. We say that $v$ is obtained from $u$ by \emph{uplifting a river} and write $u\to v$ (or $u\xrightarrow{g_i} v$ if one needs to identify the river uplifted). Thus $\to$ is a binary relation defined on the set $\lx$ of all landscapes. We illustrate the two distinct instances of this operation in Figure \ref{fig4}.

\begin{figure}[h]
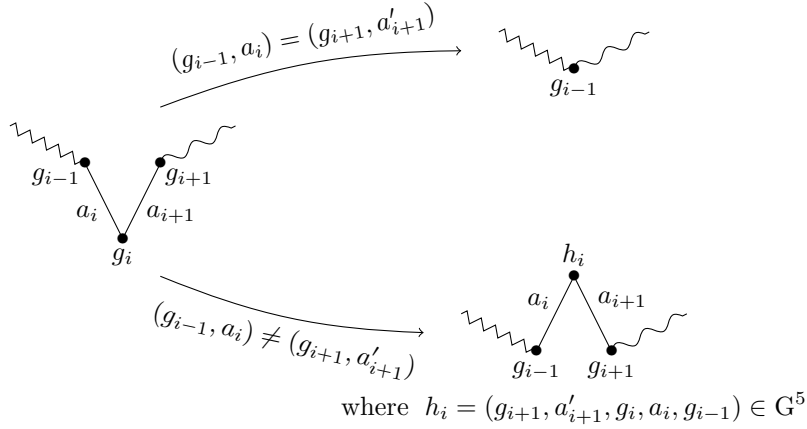

$$\tikz[scale=.5]{
\coordinate (1) at (3,3);
\coordinate (2) at (2,5);
\coordinate (3) at (0,6);
\coordinate (4) at (4,5);
\coordinate (5) at (6,6);
\draw (1) node {\small$\bullet$} node [below] {\small$g_i$};
\draw (2) node {\small$\bullet$} node [below left=-2pt] {\small$g_{i-1}$};
\draw (4) node {\small$\bullet$} node [below right=-2pt] {\small$g_{i+1}$};
\draw (2) to node[below left=-2pt] {\small$a_i$} (1);
\draw (1) to node[below right=-2pt] {\small$a_{i+1}$} (4);
\draw[snake,segment amplitude=2pt,segment length=6pt] (3)--(2);
\draw[snake=coil,segment aspect=0,segment amplitude=2pt,segment length=12pt] (4)--(5);
\coordinate (6) at (15,7.5);
\coordinate (7) at (13,8.5);
\coordinate (8) at (17,8.5);
\draw (6) node {\small$\bullet$} node [below] {\small$g_{i-1}$};
\draw[snake,segment amplitude=2pt,segment length=6pt] (7)--(6);
\draw[snake=coil,segment aspect=0,segment amplitude=2pt,segment length=12pt] (6)--(8);
\draw[->, bend left=10] (4,6.5) to node [sloped, above] {\small$(g_{i-1},a_i)=(g_{i+1},a_{i+1}')$} (12,8);
\coordinate (9) at (12,1);
\coordinate (10) at (14,0);
\coordinate (11) at (15,2);
\coordinate (12) at (16,0);
\coordinate (13) at (18,1);
\draw (10) node {\small$\bullet$} node [below] {\small$g_{i-1}$};
\draw (11) node {\small$\bullet$} node [above] {\small$h_i$};
\draw (12) node {\small$\bullet$} node [below] {\small$g_{i+1}$};
\draw (10) to node[above left=-2pt] {\small$a_i$} (11);
\draw (11) to node[above right=-2pt] {\small$a_{i+1}$} (12);
\draw[snake,segment amplitude=2pt,segment length=6pt] (9)--(10);
\draw[snake=coil,segment aspect=0,segment amplitude=2pt,segment length=12pt] (12)--(13);
\draw (15,-1.5) node {\small where $\;h_i=(g_{i+1},a_{i+1}',g_i,a_i,g_{i-1})\in \G^5$};
\draw[->, bend right=10] (4,2) to node [sloped, below] {\small$(g_{i-1},a_i)\neq(g_{i+1},a_{i+1}')$} (11,0.5);
}$$	
\caption{Illustration of the uplifting of a river.}\label{fig4}
\end{figure}

The word $v$ either keeps the length of $u$ or decreases it by 4 units. If $h_i\not\in\G^5$, the uplifting of $g_i$ eliminates the river $g_i$, but may create a new river $g_{i-1}$ if neither $g_{i-1}$ nor $g_{i+1}$ are ridges of $u$. If $h_i\in\G^5$, then the uplifting of $g_i$ replaces the river $g_i$ with a ridge $h_i$. However, $g_{i-1}$ and $g_{i+1}$ become rivers of $v$ unless they were ridges of $u$, respectively.

Let us also analyze the impact of this operation onto the `anchors subsequence' of the landscape $u$. Continuing to use the notation above, if $h_i\in\G^5$, then this anchors subsequence remains unchanged. However, if $h_i\not\in\G^5$, then we obtain a new anchors subsequence by deleting a subword of the form $xx'$, $x'x$ or $11$ from the original one.

The next result tells us that the relation $\to$ is contained in $\rho$.

\begin{lem}\label{to}
With the notation introduced above, $u\ap v$.
\end{lem}

\begin{proof}
If $h_i\not\in\G^5$, then
$$g_{i-1}a_ig_ia_{i+1}g_{i+1}=\left\{\begin{array}{ll}
g_{i-1}g_{i-1}^Lg_{i-1}&\mbox{ if }(g_i,a_{i+1})= (g_{i+1}^l,g_{i+1}^{la})\,,\\ [.2cm]
g_{i-1}g_{i-1}^Rg_{i-1}&\mbox{ if }(g_i,a_{i+1})= (g_{i+1}^r,g_{i+1}^{ra})\,;
\end{array}\right.$$
Since $g_{i-1}g_{i-1}^Lg_{i-1}\ap g_{i-1}$ and $g_{i-1}g_{i-1}^Rg_{i-1}\ap g_{i-1}$ by Lemma \ref{prodelem}.$(i)$, we conclude that $g_{i-1}a_ig_ia_{i+1}g_{i+1}\ap g_{i-1}$ and so $u\ap v$. If $h_i\in\G^5$, then
$$g_{i-1}a_ih_ia_{i+1}g_{i+1}=h_i^r\cdot h_i\cdot h_i^l\quad\mbox{ and }\quad g_{i-1}a_ig_ia_{i+1}g_{i+1}=h_i^r\cdot h_i^c\cdot h_i^l\,.$$
By definition of $\rho_s$, we conclude that $g_{i-1}a_ih_ia_{i+1}g_{i+1}\ap g_{i-1}a_ig_ia_{i+1}g_{i+1}$; whence $u\ap v$ also in this case.
\end{proof}

We will denote by $\xrightarrow{*}$ the reflexive and transitive closure of $\to$. Due to the previous result, we know that $\xr{*}$ is contained in $\rho$. Thus, if $u$ is a landscape, then
$$U(u)=\{v\in\Gp\,|\;u\xr{*}v\}\subseteq [u]\,.$$
If we assume that $u$ has length $2n+1$ and height $i$, then the landscapes from $U(u)$ all have length at most $2n+1$ and height at most $j=\lfloor i+n/2\rfloor$. They all have in common with $u$ its first and last letters. If we look to the number of rivers, we conclude they all must have at most $\lfloor n/2\rfloor$ rivers each. Hence, for each $v\in U(u)$, we record the number of rivers of $v$ in a $j$-tuple $r(v)=(r^1(v),\cdots,r^j(v))$, where $r^k(v)$ denotes the number of rivers of $v$ of height $k-1$. The sum of all entries of $r(v)$ gives us the number of rivers of $v$ and so is at most $\lfloor n/2\rfloor$. 

Let $R(u)=\{r(v)\,|\; v\in U(u)\}$. The set $R(u)$ is obviously finite. We order the elements of $R(u)$ using the lexicographic order. If $v\in U(u)$ and $v\to v_1$ by uplifting a river of height $k$, then $v_1\in U(u)$, 
$$r^k(v_1)=r^k(v)-1,\; r^{k+1}(v)\leq r^{k+1}(v_1)\leq r^{k+1}(v)+2\mbox{ and } r^{k_1}(v_1)=r^{k_1}(v)$$
for all other $k_1$; whence $r(v_1)<r(v)$. Therefore, we cannot apply upliftings of rivers indefinitely to $u$. We must stop only when we get a landscape with no river, that is, a landscape from $\ltp\cup\lo\cup \G'$. Note that these landscapes give us $j$-tuples with all entries equal to $0$. Thus $R(u)$ must always contain the $j$-tuple with all entries equal to $0$, and this $j$-tuple is obviously the smallest element of $R(u)$.

It is also easy to see that the uplifting of rivers is a commutative operation. In other words, if $g_i$ and $g_j$ are two rivers of $u\in\lx$, then the uplifting of $g_i$ followed by the uplifting of $g_j$ gives the same landscape as the uplifting of $g_j$ followed by the uplifting of $g_i$. Hence, the uplifting of rivers is a system of rules commonly known as a noetherian locally confluent system of rules. An important property of this kind of systems is that, independently of the order of the rules that we choose to apply to an element, we must always stop after a finite number of steps with the same `reduced' element (see \cite{newman}). In the case considered here, this means that, independently of the order of upliftings of rivers that we choose to apply to $u\in\lx$, we will always end up with the same landscape from $U(u)$ with no rivers. In particular, $U(u)$ has a unique landscape with no rivers. We designate it by $\be_2(u)$. Note further that $\be_2(u)=\be_2(v)$ if $u\in \lx$ and $u\xr{*} v$.

\begin{cor}\label{m2}
$\be_2(u)\ap u$ for all $u\in \lx$.
\end{cor}

\begin{proof}
This corollary is an obvious consequence of Lemma \ref{to}.
\end{proof}

Note that the set of all mountain ranges is closed for $\xr{*}$, that is, $v\in\mrx$ if $u\in\mrx$ and $u\xr{*} v$. Hence, $U(u)\subseteq\mrx$ if $u\in\mrx$. Since the only mountain ranges with no rivers are the mountains, the landscape $\be_2(u)$ must be a mountain if $u$ is a mountain range. Thus, if we set $\be(v)=\be_2(\be_1(v))$ for any word $v\in\Gp$, then $\be(v)\in\mxm$. The previous corollary and Lemma \ref{m1} also allow us to conclude that $\be(v)\ap v$. We leave this observation registered in the next corollary for future reference. 

\begin{cor}\label{m}
$\be(v)\ap v$ for all $v\in \Gp$.
\end{cor}

Let us take a pause and look more carefully to the structure of the mountains. First, they are words $u=g_0a_1g_1\cdots g_{2n-1}a_{2n}g_{2n}$ of length $4n+1$ ($n\in\mathbb{N}_0$), $g_0=g_{2n}=1$, and peak $g_n$. If $n\geq 1$, then also $g_1=g_{2n-1}=g_{xx'}\in\G_e$. In fact, if $g_i$ is the highest letter of $\la_l(u)$ belonging to $\G_e$, then the subsequence $g_0g_1\cdots g_i$ is completely determined by $g_i\,$: $g_k=g_i^{l^{i-k}}=g_i^{r^{i-k}}\in\G_e$; whence $g_0g_1\cdots g_i\in\G_e^+$. In contrast, the subsequence $a_1\cdots a_i$ may vary: $a_k$ can be any anchor from $\{1,x\}$ if $k$ even, or from $\{1,x'\}$ if $k$ odd.

If we look now to the upper part of the uphill $\la_l(u)$ (above $g_i$), quite the opposite occurs. Each $g_{k-1}$ is either the left or the right entry of $g_k$, for $i<k\leq n$. However, the side of $g_{k-1}$ inside $g_k$ completely determines $a_k$: $a_k$ is the anchor of $g_k$ corresponding to the side of $g_{k-1}$. Therefore, while the (possible empty) subsequence $g_i\cdots g_n$ may vary (although the peak $g_n$ is fixed), the subsequence $a_{i+1}\cdots a_n$ is completely determined by $g_i\cdots g_n$. Note also that $g_{i+1}\cdots g_n\in\G_d^+$ if it is not empty.

As we saw, the uphill $\la_l(u)$ is divided into two sections with very distinct properties. The same conclusions can be drawn for the downhill $\la_r(u)$. If $g_j$ is the first letter from $\G_e$ in $\la_r(u)$, then $g_k=g_j^{l^{k-j}}=g_k^{r^{k-j}}\in\G_e$ and $g_j$ fixes the subsequence $g_j\cdots g_{2n}\in\G_e^+$. In contrast, for each $k>j$, $a_k\in\{1,x\}$ if $k$ even, or $a_k\in\{1,x'\}$ if $k$ odd; whence the subsequence $a_{j+1}\cdots a_{2n}$ does not depend of $g_j$. In the upper part of $\la_r(u)$, each $g_k$ can be either the left or the right entry of $g_{k-1}$, for $n<k\leq j$. However, for $n<k\leq j$, $a_k$ is the inverse anchor of the anchor of $g_k$ determined by the side of $g_{k-1}$.

From the conclusions described above, it is now obvious that the number of uphills and downhills between $1$ and some $g\in\G_n$ is $2^n$ in both cases. We leave this conclusion registered in the following corollary for future reference.

\begin{cor}\label{counthills}
If $g\in\G_n$, then there are $2^n$ uphills from $1$ to $g$ and $2^n$ downhills from $g$ to $1$.
\end{cor}

\subsection{The word problem for $\langle\G,\R\rangle$}

In the next sequence of results we show that $\be(u)$ completely characterizes the $\rho$-class of $u\in\Gp$. We will prove that $[u]=[v]$ if and only if $\be(u)=\be(v)$, for all $u,v\in\Gp$. This result solves the word problem for $\langle\G,\R\rangle$: we just need to compute both $\be(u)$ and $\be(v)$ and check if we get the same mountain to know if $[u]=[v]$. So, let
$$\rho_1=\{(u,v)\in\G^+\times\G^+\,|\;\be(u)=\be(v)\}\,.$$
We want to show that $\rho=\rho_1$, but we already know that $\rho_1\subseteq\rho$ by Corollary \ref{m}. The next couple of results show us that $\rho_e\subseteq\rho_1$. 

\begin{lem}\label{rhoe}
Let $g\in\G'$. Then:
\begin{itemize}
\item[$(i)$] $\be_1(xx')\xr{*}\be_1(g_{xx'})$, $\be_1(xx'x)\xr{*}\be_1(x)$ and $\be_1(x'xx')\xr{*}\be_1(x')$.
\item[$(ii)$] $\la_r(g)*\la_l(g)\xrightarrow{*}g$ and $\be_1(g^2)\xr{*} \be_1(g)$.
\end{itemize}
\end{lem}

\begin{proof}
$(i)$. Clearly
$$\be_1(xx')=\be_1(x)*\be_1(x')=11g_{xx'}x1x'g_{xx'}11\xr{1}11g_{xx'}11=\be_1(g_{xx'})\,.$$
Now,
$$\be_1(xx'x)\to\be_1(g_{xx'})*\be_1(x)=11g_{xx'}111g_{xx'}x1\xr{1}11g_{xx'}x1=\be_1(x)\,.$$
We show $\be_1(x'xx')\xr{*}\be_1(x')$ similarly.

$(ii)$. The first part of $(ii)$ is proved by induction on $\up(g)$. Clearly $\la_r(1)*\la_l(1)=1$ and $\la_r(g_{xx'})*\la_l(g_{xx'})=g_{xx'}111g_{xx'}\xrightarrow{1} g_{xx'}$. Let $g\in \G_i$ for $i\geq 2$ and assume that $\la_r(h)*\la_l(h)\xrightarrow{*}h$ for all $h\in \G_j$ with $j<i$. Since
$$\la_r(g)*\la_l(g)=g(g^{ra})'g^r(g^{r_aa})'\big(\la_r(g^c)*\la_l(g^c)\big)g^{l_aa}g^lg^{la}g\,,$$
we conclude that
$$\la_r(g)*\la_l(g)\xr{*}g(g^{ra})'g^r(g^{r_aa})'g^cg^{l_aa}g^lg^{la}g$$
by the induction hypothesis. But 
$$g(g^{ra})'g^r(g^{r_aa})'g^cg^{l_aa}g^lg^{la}g\;\xr{g^c}\;g(g^{ra})'g^r g^{ra}g(g^{la})'g^lg^{la}g\xr{g^r\!,\,g^l} g\,.$$
We have shown that $\la_r(g)*\la_l(g)\xr{*} g$ and proved the first part of $(ii)$. Now, 
$$\be_1(g^2)=\la_l(g)*\la_r(g)*\la_l(g)*\la_r(g)\xr{*}\la_l(g)*g*\la_r(g)=\be_1(g)\,,$$
and we have proved the second part of $(ii)$ also.
\end{proof}

\begin{cor}\label{rhoerhoo}
$\rho_e\subseteq\rho_1$.
\end{cor}

\begin{proof}
The pairs $(xx'x,x)$, $(x'xx',x')$, $(xx',g_{xx'})$ and $(g^2,g)$, for $g\in \G'$, belong to $\rho_1$ by Lemma \ref{rhoe}. Then $\rho_e\subseteq\rho_1$ since it is also obvious that $\be(g)=\be(g1)=\be(1g)$.
\end{proof}

The next result contains technical details needed to prove that $\rho_s\subseteq\rho_1$.

\begin{lem}\label{valelem}
Let $g\in \G'$ and let $u=g_0a_1g_1\cdots a_ng_n\in \lx$. Then: 
\begin{itemize}
\item[$(i)$] If $\up(g)\geq 2$ and $s\in \{l,r\}$, then 
$\la_r(g^s)*\be_1(g^{sa})*\la_l(g)\xrightarrow{*}g^sg^{sa}g\;$ and $\;\la_r(g)*\be_1((g^{sa})')*\la_l(g^s)\xr{*}g(g^{sa})'g^s\,.$
\item[$(ii)$] $\be_1(u)\xr{*} \la_l(g_0)*u*\la_r(g_n)$.
\item[$(iii)$] $\be(u)=\be_2(u)$ if $u\in \mrx$.
\item[$(iv)$] $\be(u)=u$ if $u\in \mxm$.
\end{itemize}
\end{lem}

\begin{proof}
$(i)$. We use induction again on $\up(g)$ to prove $(i)$. We need, in fact, to prove both statements of $(i)$ simultaneously by induction on $\up(g)$. For $\up(g)=2$, recall first the two elements of $\G_2=\G_{2,e}$ described on page \pageref{G2e}. Then $g^s=g_{xx'}$, $\la_r(g^s)=g_{xx'}11$ and 
$$\la_l(g)=\left\{\begin{array}{ll}
11g_{xx'}1g_{2,e,1}&\mbox{ if } g=g_{2,e,1}\,, \\ [.2cm]
1x'g_{xx'}xg_{2,e,2}&\mbox{ if } g=g_{2,e,2}\,.
\end{array}\right.$$
Consider first the case $s=l$. If $g=g_{2,e,1}$, then $\be_1(g^{sa})=\be_1(1)=1$ and
$$\la_r(g^s)*\be_1(g^{sa})*\la_l(g)\;=\;g_{xx'}111g_{xx'}1g\;\xr{1}\;g_{xx'}1g\;=\;g^sg^{sa}g\,.$$
If $g=g_{2,e,2}$, then $\be_1(g^{sa})=\be_1(x)=11g_{xx'}x1$ and
$$\la_r(g^s)*\be_1(g^{sa})*\la_l(g)\, =\, g_{xx'}111g_{xx'}x1x'g_{xx'}xg\; \xr{1,\,1}\;g_{xx'}xg \,=\, g^sg^{sa}g\,.$$
Now, consider the case $s=r$. If $g=g_{2,e,1}$, then $\be_1(g^{sa})=\be_1(x)=11g_{xx'}x1$ and
$$\begin{array}{ll}
\la_r(g^s)*\be_1(g^{sa})*\la_l(g)\hspace*{-.1cm}&=\;g_{xx'}111g_{xx'}x11g_{xx'}1g\; \xr{1}\; g_{xx'}x11g_{xx'}1g\;\xr{1}\\ [.2cm]
&\xr{1}\;g_{xx'}xg1g_{xx'}1g \;\xr{g_{xx'}}\;g_{xx'}xg\;=\; g^sg^{sa}g\,.
\end{array}$$
If $g=g_{2,e,2}$, then $\be_1(g^{sa})=\be_1(1)=1$ and
$$\begin{array}{ll}
\la_r(g^s)*\be_1(g^{sa})*\la_l(g)\hspace*{-.1cm}&=\;g_{xx'}11x'g_{xx'}xg\;\xr{1}\\ [.2cm]
&\xr{1}\;g_{xx'}1gx'g_{xx'}xg\;\xr{g_{xx'}}\;g_{xx'}1g\;=\;g^sg^{sa}g\,.
\end{array}$$
We have proved the statement $\la_r(g^s)*\be_1(g^{sa})*\la_l(g)\xrightarrow{*}g^sg^{sa}g$ for  $\up(g)=2$. Similarly, we prove the statement $\la_r(g)*\be_1((g^{sa})')*\la_l(g^s)\xr{*}g(g^{sa})'g^s$ for $\up(g)=2$.

Let $g\in \G_i$ with $i>2$, and assume that $\la_r(h^t)*\be_1(h^{ta})*\la_l(h)\xrightarrow{*}h^th^{ta}h$ and $\;\la_r(h)*\be_1((h^{ta})')*\la_l(h^t)\xr{*}h(h^{ta})'h^t$ for all $h\in\G_j$ with $2\leq j<i$ and $t\in\{l,r\}$. Let $s\in\{l,r\}$ and recall that $g^{sa}=(g^{s_aa})'$ and $g^{s_a}=g^c$. Then
$$\begin{array}{ll}
\la_r(g^s)*\be_1(g^{sa})*\la_l(g)\hspace*{-.1cm}&=\;\la_r(g^s)*\be_1(g^{sa})*\la_l(g^c)g^{l_aa}g^lg^{la}g \\ [.2cm]
&=\;\la_r(g^s)*\be_1((g^{s_aa})')*\la_l(g^c)g^{l_aa}g^lg^{la}g  \\ [.2cm]
&\xr{*}\;g^s(g^{s_aa})'g^cg^{l_aa}g^lg^{la}g \,,
\end{array}$$
where $\xr{*}$ follows by the induction hypothesis. Note now that if $s=l$, then
$$g^s(g^{s_aa})'g^cg^{l_aa}g^lg^{la}g=g^l(g^{l_aa})'g^cg^{l_aa}g^lg^{la}g\xr{g^c}g^lg^{la}g=g^sg^{sa}g\,;$$
and if $s=r$, then 
$$\begin{array}{lcl}
g^s(g^{s_aa})'g^cg^{l_aa}g^lg^{la}g&=&g^r(g^{r_aa})'g^cg^{l_aa}g^lg^{la}g\\ [.2cm] 
&\xr{g^c}&g^r(g^{r_aa})'gg^{l_aa}g^lg^{la}g \\ [.2cm]  
&=&g^rg^{ra}g(g^{la})'g^lg^{la}g \\ [.2cm]
&\xr{g^l}&  g^rg^{ra}g\;=\;g^sg^{sa}g\,.
\end{array}$$
We have proved that $\la_r(g^s)*\be_1(g^{sa})*\la_l(g)\xrightarrow{*}g^sg^{sa}g$. The proof of $\la_r(g)*\be_1((g^{sa})')*\la_l(g^s)\xr{*}g(g^{sa})'g^s$ is similar, and so $(i)$ is proved by mathematical induction.

$(ii)$. For $i\in\{0,\cdots,n\}$, let $u_i=g_0a_1g_1\cdots a_ig_i$. Thus all $u_i$ are prefixes of $u$ and $u=u_n$. Consequently, $(ii)$ becomes proved once we show that $\be_1(u_i)\xr{*} \la_l(g_0)*u_i*\la_r(g_i)$ by induction on $i$. Clearly 
$$\be_1(u_0)=\la_l(g_0)*u_0*\la_r(g_0)$$
since $u_0=g_0$. Assume that $\be_1(u_{i-1})\xr{*} \la_l(g_0)*u_{i-1}*\la_r(g_{i-1})$. By definition of landscape and $(i)$, we have also 
$$\la_r(g_{i-1})*\be_1(a_i)*\la_l(g_i)\xr{*}g_{i-1}a_ig_i\,.$$
Now, using the induction hypothesis, we obtain
$$\begin{array}{ll}
\be_1(u_i)\hspace*{-.1cm}& =\;\be_1(u_{i-1})*\be_1(a_i)*\be_1(g_i)\\ [.2cm]
& \xr{*}\; \la_l(g_0)*u_{i-1}*\la_r(g_{i-1})*\be_1(a_i)*\la_l(g_i)*\la_r(g_i) \\ [.2cm]
& \xr{*}\; \la_l(g_0)*u_{i-1}*(g_{i-1}a_ig_i)*\la_r(g_i)=\la_l(g_0)*u_{i}*\la_r(g_i)\,.
\end{array}$$
We have finished the proof of $(ii)$.

$(iii)$ follows easily from $(ii)$ since $g_0=1=g_n$ if $u\in\mrx$:
$$\be(u)=\be_2(\be_1(u))=\be_2(\la_l(g_0)*u*\la_r(g_n))=\be_2(u)\,;$$
and $(iv)$ is obvious from $(iii)$.
\end{proof}

\begin{prop}\label{mrhoinv}
Let $g\in \G^5$ with $\up(g)\geq 2$ and $u,v\in \Gp$.
\begin{itemize}
\item[$(i)$] $\be(g)=\be(g^cg^Lg)=\be(gg^Rg^c)\;$ and $\;\be(g^r\!\!\cdot\! g^c\!\cdot\! g^l)=\be(g^r\!\!\cdot\! g\!\cdot\! g^l)$.
\item[$(ii)$] $\be(u)=\be(v)$ if and only if $[u]=[v]$.
\end{itemize}
\end{prop}

\begin{proof}
$(i)$. By Lemma \ref{valelem}.$(ii)$, 
$$\be_1(g^cg^Lg) \xr{*} \la_l(g^c)*(g^cg^Lg)*\la_r(g)=\la_l(g)*\la_r(g)=\be_1(g)\,.$$ 
Thus $\be(g^cg^Lg)=\be(g)$. We have also $\be(gg^Rg^c)=\be(g)$ by duality. Again by Lemma \ref{valelem}.$(ii)$, we have $\be_1(g^r\cdot g\cdot g^l)\xr{*}\la_l(g^r)*(g^r\cdot g\cdot g^l)*\la_r(g^l)$ and
$$\begin{array}{ll}
\be_1(g^r\cdot g^c\cdot g^l)\hspace*{-.1cm}& \xr{*}\;\la_l(g^r)*(g^r(g^{r_aa})'g^cg^{l_aa}g^l)*\la_r(g^l) \\ [.2cm]
& \xr{g^c}\; \la_l(g^r)*(g^r(g^{r_aa})'gg^{l_aa}g^l)*\la_r(g^l) \\ [.2cm]
& =\; \la_l(g^r)*(g^rg^{ra}g(g^{la})'g^l)*\la_r(g^l) \\ [.2cm]
& =\; \la_l(g^r)*(g^r\cdot g\cdot g^l)*\la_r(g^l)\,.
\end{array}$$
Thereby $\be(g^r\cdot g^c\cdot g^l)=\be(g^r\cdot g\cdot g^l)$.

$(ii)$. We already know that $[u]=[v]$ if $\be(u)=\be(v)$ by Corollary \ref{m}. From Corollary \ref{rhoerhoo}, we have also $\rho_e\subseteq\rho_1$. Now, we proved in $(i)$ that $\rho_s\subseteq\rho_1$; whence $R\subseteq\rho_1$. To finish the proof of $(ii)$, we just need to conclude that $\rho_1$ is a congruence on $\G^+$. But this is obvious since $\be(u_1u_2)=\be(\be(u_1)u_2)=\be(u_1\be(u_2))$.
\end{proof}

Although $\G$ and $\rho_e\cup\rho_s$ are infinite sets, Proposition \ref{mrhoinv}.$(ii)$ gives us a solution for the word problem for $\foo$ as explained before: to see if $u\ap v$, we just need to compute both $\be(u)$ and $\be(v)$, and check if we get the same mountain. Thus:

\begin{cor}
The word problem for $\foo$ is decidable.
\end{cor}

Proposition \ref{mrhoinv}.$(ii)$ tells us also that each $\rho$-class $[u]$ has a unique mountain, namely $\be(u)$. We will call $\be(u)$ the \emph{canonical form} of $u\in \Gp$. We introduce the operation $\odot$ on both $\mxm$ and $\mx$ as follows:
$$u_1\odot u_2= \be(u_1*u_2)=\be_2(u_1*u_2)\,,$$
for any $u_1,u_2\in\mxm$. 

\begin{prop}\label{model}
$(\mxm,\odot)$ and $(\mx,\odot)$ are models for $\foo$ and $\fo$, respectively. Moreover, if $u,v,w\in \mxm$ are such that $w=u\odot v$, then
\begin{itemize}
\item[$(i)$] $\la_l(u)$ is a prefix of $\la_l(w)$, while $\la_r(v)$ is a suffix of $\la_r(w)$.
\item[$(ii)$] $\up(w)\geq\max\{\up(u),\up(v)\}$, and $\up(w)=\up(u)$ {\rm [}$\up(w)=\up(v)${\rm ]} \iff\ $\ka(w)=\ka(u)$ {\rm [}$\ka(w)=\ka(v)${\rm ]}.
\end{itemize}
\end{prop}

\begin{proof}
The first part is just a consequence of  Proposition \ref{mrhoinv}.$(ii)$, while $(ii)$ follows obviously from $(i)$. Hence, we only need to prove $(i)$. But notice that, by definition of uplifting of rivers, if $u_1\to u_2$ for a mountain range $u_1$, then $\la_l(u_1)$ is a prefix of $\la_l(u_2)$ and $\la_r(u_1)$ is a suffix of $\la_r(u_2)$. Applying several times the previous observation, we conclude that $\la_l(u)=\la_l(u*v)$ is a prefix of $\la_l(\be(u*v))=\la_l(u\odot v)$, while $\la_r(v)=\la_r(u*v)$ is a suffix of $\la_r(\be(u*v))=\la_r(u\odot v)$.
\end{proof}

\section{The semigroup $\fo$}

In the previous section we constructed $\fo$. In this section we prove that $\fo$ is a regular semigroup weakly generated by $[x]$. We begin by showing that $\fo$ is regular. To prove that $\fo$ is weakly generated by $[x]$, we need first to describe the Green's relations on $\fo$ and then prove that $[g]$ is the only element in $S([g^Rg^c],[g^cg^L])$. This latter fact will imply that any regular subsemigroup of $\fo$ containing $[x]$, will contain also $\G^5$. Using this information, we will conclude that $\fo$ is weakly generated by $[x]$.

Let $u=g_0a_1g_1\cdots g_{n-1}a_ng_n$ be a landscape. Note that $g_na_ng_{n-1}\cdots g_1a_1g_0$ is not a landscape in general but $g_na_n'g_{n-1}\cdots g_1a_1'g_0$ is. We define the \emph{reverse} of $u$ as the landscape $\cev{u}=g_na_n'g_{n-1}\cdots g_1a_1'g_0$. Clearly the reverse of $\cev{u}$ is $u$ again.

\begin{lem}\label{updownhill}
If $u=g_0a_1g_1\cdots a_ng_n\in \lom$, then $u\cev{u}\ap u*\cev{u}\ap g_0$.
\end{lem}

\begin{proof}
Note that $u*\cev{u}$ is a valley with river $g_n$ and clearly $u\cev{u}\ap u*\cev{u}$ since $[g_n]$ is an idempotent of $\foo$. Let $i\in\{1,\cdots,n\}$. Then $g_i=g_{i-1}^{s_i}$ for some $s_i\in\{l,r\}$ since $u\in\lom$. In fact, we know also that $a_i=(g_{i-1}^{s_ia})'$. Hence 
$$g_{i-1}a_ig_ia_i'g_{i-1}=g_{i-1}\cdot g_{i-1}^{s_i}\cdot g_{i-1}\ap g_{i-1}$$ 
by Lemma \ref{prodelem}.$(i)$. Now, applying several times this observation, we obtain
$$\begin{array}{ll}
u\cev{u} \ap u*\cev{u}\!\!\! & = g_0a_1\cdots g_{n-2}a_{n-1}g_{n-1}a_ng_na_n'g_{n-1}a_{n-1}'g_{n-2}\cdots a_1'g_0 \\ [.2cm]
& \ap g_0a_1\cdots g_{n-2}a_{n-1}g_{n-1}a_{n-1}'g_{n-2}\cdots a_1'g_0 \\ 
&\hspace*{1.3cm} \vdots \\
& \ap g_0a_1g_1a_1'g_0 \ap g_0\,.
\end{array}$$
We have proved that $u\cev{u}\ap u*\cev{u}\ap g_0$ as wanted.
\end{proof}

Next, we prove that $\foo$ is a regular monoid.

\begin{prop}\label{reg}
The monoid $\foo$ is regular and $[\cev{u}]$ is an inverse of $[u]$ for all $u\in \ltp$.
\end{prop}

\begin{proof}
By Corollary \ref{m}, $v\ap \be(v)$ for all $v\in \Gp$. Since $\be(v)\in \mxm\subseteq \ltp\cup\{1\}$, it is enough to prove the second part of this proposition to conclude also the first. Let $u\in\ltp$ with peak $g\in\G^5$. Then $u=u_1*u_2$ for some $u_1\in\lop$ and $u_2\in\lom$ such that $\tau(u_1)=g=\sigma(u_2)$. By Lemma \ref{updownhill}, $u_2\cev{u_2}\ap g\ap\cev{u_1}u_1$ (note that $\cev{u_1}\in\lom$ and its reverse is $u_1$). Hence
$$u\cev{u}u=u_1*u_2\,\cev{u_2}*\cev{u_1}\,u_1*u_2\ap u_1*g*g*u_2=u_1*u_2=u\,.$$
We show similarly that $\cev{u}u\cev{u}\ap \cev{u}$. Therefore, $[\cev{u}]$ is an inverse of $[u]$.
\end{proof}

The \emph{ground} $\ep(g)$ of a letter $g\in \G'$ is defined recursively as follows: 
$$\ep(1)=\{1\}\quad\mbox{ and }\quad\ep(g)=\ep(g^l)\cup\{g\} \cup\ep(g^r)$$ 
if $\up(g)\geq 1$. Thus $\up(g_1)<\up(g)$ for any $g_1\in\ep(g)\setminus\{g\}$. We define the relation $\preceq$ on $\G'$ by setting $h\preceq g$ if $h\in\ep(g)$. Due to statement $(i)$ of the following result, $\preceq$ is a partial order on $\G'$. We extend the notion of ground to any landscape by setting, for $u=g_0a_1g_1\cdots a_ng_n\in \lx$, $\ep(u)=\cup_{i=0}^n\ep(g_i)$. Again by statement $(i)$ of the following result, $\ep(u)$ is the union of the grounds of its ridges, and the highest letters of $\ep(u)$ are the peaks of $u$. In particular, $\ep(u)=\ep(\ka(u))$ if $u$ is a mountain.

\begin{lem}\label{ground}
Let $g,h\in \G'$ and $u,v,w\in\mxm$ such that $w=u\odot v$. Then:
\begin{itemize}
\item[$(i)$] $h\in\ep(g)$ \iff\ $\ep(h)\subseteq\ep(g)$.
\item[$(ii)$] If $h\in\ep(g)\setminus\{g\}$, then there exists $u=g_0a_1g_1\cdots a_ng_n\in \lop$ such that $n=\up(g)-\up(h)$, $g_i\in\ep(g)$ for all $0\leq i\leq n$, $g_0=h$ and $g_n=g$.
\item[$(iii)$] $\ep(u)\cup\ep(v)\subseteq\ep(w)$.
\end{itemize}
\end{lem} 

\begin{proof}
$(i)$. Of course, we only need to prove that $\ep(h)\subseteq\ep(g)$ if $h\in\ep(g)$. This statement is proved by induction on $\up(g)$, and it is trivially true for $\up(g)=0$, that is, for $g=1$. Assume that $\ep(h_1)\subseteq\ep(g_1)$ for all $g_1\in\G'$ such that $\up(g_1)<\up(g)$ and all $h_1\in\ep(g_1)$. If $h\in\ep(g)$, then either $h=g$, or $h\in\ep(g^l)$, or $h\in\ep(g^r)$. In the latter two cases, we have either $\ep(h)\subseteq\ep(g^l)$ or $\ep(h)\subseteq\ep(g^r)$, respectively, by the induction hypothesis. Hence,
$$\ep(h)\subseteq \ep(g^r)\cup \{g\}\cup\ep(g^l)=\ep(g)$$
as wanted.

$(ii)$. Let $h\in\ep(g)\setminus\{g\}$. We prove $(ii)$ by induction on $n=\up(g)-\up(h)$. If $n=1$, then $h=g^s$ for some $s\in\{l,r\}$. Hence $u=hg^{sa}g$ is an uphill satisfying the conditions of $(ii)$. Assume now that $n\geq 2$. Then $h\in\ep(g^s)$ for some $s\in\{l,r\}$. Without loss of generality, we assume that $h\in\ep(g^l)$. Since $\up(g^l)-\up(h)=n-1$, there exists $g_0a_1g_1\cdots a_{n-1}g_{n-1}\in\lop$ such that all $g_i\in\ep(g^l)$ for $0\leq i\leq n-1$, $g_0=h$ and $g_{n-1}=g^l$, by the induction hypothesis. Clearly $u=g_0a_1g_1\cdots a_{n-1}g_{n-1}g^{la}g$ is now an uphill satisfying the conditions stated in $(ii)$.

$(iii)$. Note that $\ka(u),\ka(v)\in\ep(w)$ since $\la_l(u)$ is a prefix of $\la_l(w)$ and $\la_r(v)$ is a suffix of $\la_r(w)$ (Proposition \ref{model}.$(i)$). $(iii)$ follows now from $(i)$. 
\end{proof}

We describe next the partial orders $\leq_{\Rc}$, $\leq_{\Lc}$ and $\leq_{\Jc}$ on $(M^1,\odot)$.

\begin{prop}\label{desR}
Let $u,v\in \mxm$. Then:
\begin{itemize}
\item[$(i)$] $u\leq_{\Rc} v$ \iff\ $\la_l(v)$ is a prefix of $\la_l(u)$. 
\item[$(ii)$] $u\leq_{\Lc} v$ \iff\ $\la_r(v)$ is a suffix of $\la_r(u)$. 
\item[$(iii)$] $u\leq_{\Jc} v$ \iff\ $\ka(v)\preceq\ka(u)$. 
\end{itemize}
\end{prop}

\begin{proof}
We prove only $(i)$ and $(iii)$ since $(ii)$ is the dual of $(i)$.

$(i)$. Assume that $u\leq_{\Rc} v$. Thus $u=v\odot w$ for some $w\in\mxm$. By Proposition \ref{model}.$(i)$, $\la_l(v)$ is a prefix of $\la_l(u)$. Assume now that $\la_l(v)$ is a prefix $\la_l(u)$. Then $u=\la_l(v)*u_1$ for some landscape $u_1$ with $\tau(\la_l(v))=\ka(v)=\si(u_1)$. Let $v_1=\la_r(v)$ and consider $w=\cevm{v_1}*u_1$. Clearly $w\in\mxm$. Further, $v_1*\cevm{v_1}\xr{*} \ka(v)$ by Lemma \ref{updownhill}, and so
$$v*w=\la_l(v)*v_1*\cevm{v_1}*u_1\xr{*}\la_l(v)*\ka(v)*u_1=\la_l(v)*u_1=u\,,$$
that is, $u=v\odot w$ and $u\leq_{\Rc} v$. 

$(iii)$. If $u\leq_{\Jc} v$, then $u=w_1\odot v\odot w_2$ for two mountains $w_1$ and $w_2$. By Lemma \ref{ground}.$(iii)$, $\ep(\ka(v))=\ep (v)\subseteq\ep(u)=\ep(\ka(u))$ and $\ka(v)\preceq\ka(u)$. Conversely, if $\ka(v)\preceq\ka(u)$, then let $u_1$ be an uphill such that $\si(u_1)=\ka(v)$ and $\tau(u_1)=\ka(u)$, whose existence is guaranteed by Lemma \ref{ground}.$(ii)$. Observe that $w_1=\la_l(u)*\cevm{u_1}*\cevl{\la_l(v)}$ and $w_2=\cevl{\la_r(v)}*u_1*\la_r(u)$ are well defined mountains. Further,
$$\begin{array}{ll}
w_1*v*w_2\!\!\! & =\la_l(u)*\cevm{u_1}*(\cevl{\la_l(v)}*\la_l(v))*(\la_r(v)*\cevl{\la_r(v)})*u_1*\la_r(u) \\ [.2cm]
& \xr{*} \la_l(u)*\cevm{u_1}*\ka(v)*\ka(v)*u_1*\la_r(u) \\ [.2cm]
& \xr{*} \la_l(u)*\ka(u)*\la_r(u)=u\,,
\end{array}$$
and $w_1\odot v\odot w_2=u$. Thus $u\leq_{\Jc} v$.
\end{proof}

The next two results are corollaries of the previous proposition. The first one is an obvious consequence. For the second one, we need to say something more about the $\Hc$ and the $\Dc$ relations.

\begin{cor}\label{covers}
Let $u,v\in \mxm$. Then:
\begin{itemize}
\item[$(i)$]  $v$ covers $u$ for $\leq_{\Rc}$ \iff\ $\la_l(u)=\la_l(v)a\ka(u)$ for some $a\in A$.
\item[$(ii)$] $v$ covers $u$ for $\leq_{\Lc}$ \iff\ $\la_r(u)=\ka(u)a\la_r(v)$ for some $a\in A$.
\item[$(iii)$] $v$ covers $u$ for $\leq_{\Jc}$ \iff\ $\ka(v)\in\{(\ka(u))^l,(\ka(u))^r\}$.
\end{itemize}
\end{cor}

\begin{cor}\label{R}
Let $u,v\in\mxm$. Then:
\begin{itemize}
\item[$(i)$] $u\Rc v$ \iff\ $\la_l(u)=\la_l(v)$.
\item[$(ii)$] $u\Lc v$ \iff\ $\la_r(u)=\la_r(v)$.
\item[$(iii)$] $u\Jc v$ \iff\ $\ka(u)=\ka(v)$.
\item[$(iv)$] $u\Hc v$ \iff\ $u=v$.
\item[$(v)$] $\Dc =\Jc$.
\end{itemize}
\end{cor}

\begin{proof}
The first three statements follow from the corresponding statements of Proposition \ref{desR}. Then $(iv)$ is a consequence of $(i)$ and $(ii)$. So, we only need to prove that $\Jc\subseteq\Dc$. Assume that $u\Jc v$. Then $w=\la_l(u)*\la_r(v)\in \mxm$ since $\ka(u)=\ka(v)$ by $(iii)$. Now, by $(i)$ and $(ii)$, we conclude that $u\Rc w\Lc v$, whence $u\Dc v$.
\end{proof}

The statements $(iii)$ and $(v)$ of the previous corollary tell us that the $\Dc=\Jc$-classes of $\foo$ are in one-to-one correspondence with the elements of $\G'$, and that the set $\{[g]\,|\;g\in\G'\}$ is a transversal (or cross-section) for the set of $\Dc$-classes of $\foo$. The next result gives us the size of each $\Rc$, $\Lc$ and $\Dc$-class of $\foo$.

\begin{cor}\label{size}
If $g\in \G_n$ for $n\geq 1$, then $|\Rcc_{[g]}|=2^n=|\Lcc_{[g]}|$ and $|\Dcc_{[g]}|=2^{2n}$. Further, $\Dcc_{[g]}$ has $2^n$ $\Rc$-classes and $2^n$ $\Lc$-classes. 
\end{cor}

\begin{proof}
We only need to prove that $|\Rcc_{[g]}|=2^n$ since $|\Lcc_{[g]}|=2^n$ follows by duality and the statements about $\Dc$ follow from the statements about $\Rc$ and $\Lc$ and from Corollary \ref{R}.$(iv)$. But Corollary \ref{R}.$(i)$ tells us that the size of $\Rcc_{[g]}$ is equal to the number of downhills from $g$ to $1$, and this number is $2^n$ due to Corollary \ref{counthills}. Hence, $|\Rcc_{[g]}|=2^n$.
\end{proof}

We know that $[g]\in S([g^Rg^c],[g^cg^L])$ by Lemma \ref{prodelem}.$(vi)$. In the next result we prove that $[g]$ is the only element of $S([g^Rg^c],[g^cg^L])$ in $\foo$.

\begin{lem}\label{sandwich}
For each $g\in \G^5$ with $\up(g)\geq 2$, $S([g^Rg^c],[g^cg^L])=\{[g]\}\,$ in $\foo$.
\end{lem}

\begin{proof}
Let $u\in\mxm$ be such that $[u]\in S([g^Rg^c],[g^cg^L])$. Then $\ka(u)=g$ by Corollary \ref{R} since $[u]\Dc [g]$. Let $v$ be the mountain $\la_l(g^c)g^{l_aa}\la_r(g^l)$. By Lemma \ref{valelem}.$(ii)$, $\be_1(g^cg^{l_aa}g^l)\xr{*} v$, and so $g^cg^{l_aa}g^l\ap v$. Hence
$$u\ap g^cg^Lu\ap vg^{la}u\ap v*\be_1(g^{la})* u\ap v\odot\be_1(g^{la})\odot u\,,$$
and so $u=v\odot\be_1(g^{la})\odot u\,$. By Proposition \ref{model}.$(i)$, $\la_l(v)=\la_l(g^c)g^{l_aa}g^l$ is a prefix of $\la_l(u)$. In a similar way, we conclude also that $g^r(g^{r_aa})'\la_r(g^c)$ is a suffix of $\la_r(u)$.

If $g\in\G_d$, then $g^l\neq g^r$. Hence, for this case, $\la_l(g)$ is the only uphill from $1$ to $g$ with prefix $\la_l(g^c)g^{l_aa}g^l$ and $\la_r(g)$ is the only downhill from $g$ to $1$ with suffix $g^r(g^{r_aa})'\la_r(g^c)$. We have concluded that $\la_l(g)=\la_l(u)$ and $\la_r(g)=\la_r(u)$ since $\ka(u)=g$. Thus, $u=\be_1(g)$ and $[u]=[g]$ if $g\in\G_d$. 

The case $g\in\G_e$ is more complex since $g^l=g^r$, and thus there are two distinct uphills from $1$ to $g$, namely $\la_l(g)$ and $\la_l(g^c)g^{l_aa}g^lg^{ra}g$. There are also two distinct downhills form $g$ to $1$, namely $\la_r(g)$ and $g(g^{la})'g^r(g^{r_aa})'\la_r(g^c)$. We have concluded that 
$$u=\la_l(g^c)g^{l_aa}g^la_1ga_2'g^r(g^{r_aa})'\la_r(g^c)\,,$$
for $a_1,a_2\in\{g^{ra},g^{la}\}$. However, $[u]$ is also an idempotent of $\foo$. Hence,
$$\begin{array}{ll}
u\ap u^2\hspace*{-.1cm}& =\; \la_l(u)*\la_r(u)*\la_l(u)*\la_r(u)\\ [.2cm]
&\xr{*}\; \la_l(u)*ga_2'g^r(g^{r_aa})'g^cg^{l_aa}g^la_1g*\la_r(u) \\ [ .2cm]
&\xr{g^c}\; \la_l(u)*ga_2'g^rg^{ra}g(g^{la})'g^la_1g*\la_r(u) \,.
\end{array}$$
If $a_2=g^{la}$, then $g_1=(g,(g^{ra})',g^r,(g^{la})',g)\in\G_e$ and 
$$ga_2'g^rg^{ra}g\xr{g^r} ga_2'g_1g^{ra}g\,.$$ 
So, $g_1\in\ep(\be(u^2))=\ep(u)=\ep(g)$ and $\up(g_1)=\up(g)+1$, which is a contradiction. Therefore $a_2$ must be equal to $g^{ra}$. Similarly, we conclude also that $a_1=g^{la}$. Consequently $\la_l(u)=\la_l(g)$ and $\la_r(u)=\la_r(g)$, that is, $u=\be_1(g)$ and $[u]=[g]$ if $g\in\G_e$ also. We have proved the statement of this proposition.
\end{proof}

We have now all the ingredients necessary to show that $\fo$ is weakly generated by $[x]$.

\begin{prop}\label{weakly}
The regular semigroup $\fo$ is weakly generated by $[x]$.
\end{prop}

\begin{proof}
Let $S$ be a regular subsemigroup of $\fo$ containing $[x]$. By Corollary \ref{R},
$$\Dcc_{[x]}=\{[x],[x'],[g_{xx'}],[x'x]\}\,,$$
and $[x']$ is the only inverse of $[x]$. Hence, $S$ contains $\Dcc_{[x]}$. If $g\in\G_2$, then $g^c=1$ and $[g]$ is the only element of $S([g^Rg^c],[g^cg^L])= S([g^R],[g^L])$. Note that both $[g^R]$ and $[g^L]$ belong to $S$ since $g^r=g^l=g_{xx'}$. Since $S$ is regular, it must contain also an element from $S([g^R],[g^L])$, that is, it must contain $[g]$. Next, we show that $S$ contains the set $\{[g]\,|\;g\in\G^5\}$ by induction on $\up(g)$.

Let $g\in\G_i$ for $i\geq 3$, and assume that $[h]$ belongs to $S$ for all $h\in\G_j$ with $j<i$. Then $[g^l]$, $[g^r]$ and $[g^c]$ belong to $S$ by the induction hypothesis (note that $g^c\neq 1$ because $\up(g)\geq 3$). Since $[x]$ and $[x']$ also belong to $S$, we conclude that $[g^Rg^c]\in S$ and $[g^cg^L]\in S$. Again, we know that
$$S([g^Rg^c],[g^cg^L])\cap S\neq\emptyset$$
because $S$ is regular; whence $[g]\in S$ due to Lemma \ref{sandwich}. We have shown that $S$ contains
$$\{[x],[x']\}\cup\{[g]\,|\;g\in\G^5\}\,.$$
Since this latter set generates $\fo$, we must have $S=\fo$. Therefore, $\fo$ is weakly generated by $[x]$.
\end{proof}

In this section we have already described the Green's relations on $\fo$. We end it describing the idempotents, the inverses and the natural partial order on $\fo$. A \emph{canyon} is a valley $w$ with $\si(w)=\tau(w)$ and a \emph{gorge} is a canyon such that $w\xr{*}\si(w)=\tau(w)$. Thus $\be_2(w)=\si(w)$ for any gorge $w$, and $w_1\xr{*}\si(w)$ for any $w_1\in\G^+$ such that $w\xr{*} w_1$. Note however that $w_1$ may not be a gorge since we cannot guarantee that $w_1$ is even a valley. We will use gorges to characterize the idempotents, the inverses and the natural partial order, but this characterization will be just theoretical. For practical purposes, it is hard to identify the gorges, and consequently also the idempotents, the inverses and the natural partial order.

\begin{prop}\label{idgorges}
Let $u,v\in\mxm$.
\begin{itemize}
\item[$(i)$] $u$ is an idempotent \iff\ the canyon $\la_r(u)*\la_l(u)$ is a gorge.
\item[$(ii)$] $v$ is an inverse of $u$ \iff\ $\la_r(u)*\la_l(v)$ and $\la_r(v)*\la_l(u)$ are both gorges.
\item[$(iii)$] $v<u$ \iff\ $\la_l(v)=\la_l(u)\,a_1\,u_1$ and $\la_r(v)=u_2\,a_2\, \la_r(u)$ for some $u_1\in\lop\cup\G^5$, $u_2\in\lom\cup\G^5$ and $a_1,a_2\in A$ such that $u_2\,a_2\,\ka(u)\,a_1\,u_1$ is a gorge.
\end{itemize}
\end{prop}

\begin{proof}
If $\la_r(u)*\la_l(u)\xr{*} \ka(u)$, then
$$u^2\ap \la_l(u)*\la_r(u)*\la_l(u)*\la_r(u)\xr{*}\la_l(u)*\ka(u)*\la_r(u)=u$$
and $u$ is an idempotent. Conversely, if $u$ is an idempotent, then $u^2\xr{*} u$. By definition of uplifting of rivers, it is obvious that we must have $\la_r(u)*\la_l(u)\xr{*} \ka(u)$. We have proved $(i)$. The proof of $(ii)$ is similar and as obvious as the proof of $(i)$. Let us prove $(iii)$ now. 
	
Assume first that $v<u$. In particular $v\leq_{\Rc} u$ and $v\leq_{\Lc} u$. Since $u\neq v$,
$$\la_l(v)=\la_l(u)\,a_1\,u_1\quad\mbox{ and }\quad\la_r(v)=u_2\,a_2\,\la_r(u)\,,$$
for some $u_1\in\lop\cup\G^5$, $u_2\in\lom\cup\G^5$ and $a_1,a_2\in A$, by Proposition \ref{desR}. But, there exists also another mountain $w\in\mxm$ such that
$$w\in E(\mxm),\quad w\Lc v\quad\mbox{ and }\quad v=u\odot w$$
as $v< u$. Hence, $\la_r(w)*\la_l(w)$ is a gorge by $(i)$, $\la_r(w)=\la_r(v)$ by Corollary \ref{R}.$(ii)$, and $\la_r(u)*\la_l(w)\xr{*} \ka(u)\,a_1\,u_1$. Then
$$\la_r(w)*\la_l(w)=\la_r(v)*\la_l(w)=u_2\,a_2\,\la_r(u)*\la_l(w)\xr{*} u_2\,a_2\,\ka(u)\,a_1\,u_1\,.$$
We can now conclude that $u_2\,a_2\,\ka(u)\,a_1\,u_1$ is a gorge because $\la_r(w)*\la_l(w)$ is a gorge too (note that $u_2\,a_2\,\ka(u)\,a_1\,u_1$ is a canyon).
	
Assume now that $\la_l(v)=\la_l(u)\,a_1\,u_1$ and $\la_r(v)=u_2\,a_2\, \la_r(u)$ for some $u_1\in\lop\cup\G^5$, $u_2\in\lom\cup\G^5$ and $a_1,a_2\in A$ such that $u_2\,a_2\,\ka(u)\,a_1\,u_1$ is a gorge. In particular $u\neq v$. Note also that
$$w_r=\cevl{\la_r(u)}\,a_1\,u_1*\la_r(v)\quad\mbox{ and }\quad w_l=\la_l(v)*u_2\,a_2\,\cevl{\la_l(u)}$$
are well defined mountains. In fact, they are both idempotents of $\mxm$. For example,
$$\la_r(v)*\cevl{\la_r(u)}\,a_1\,u_1 =u_2\,a_2\,\la_r(u)*\cevl{\la_r(u)}\,a_1\,u_1\xr{*}u_2\,a_2\, \ka(u)\,a_1\, u_1\,;$$
thus $\la_r(v)*\cevl{\la_r(u)}\,a_1\,u_1$ is a gorge and so $w_r$ is an idempotent. Finally, observe that $u\odot w_r=v$ and $w_l\odot u=v$, whence $v<u$.
\end{proof}

\section{The universal property of $\fo$}\label{sec5}
 
In this section we prove that all regular semigroups weakly generated by $x$ are homomorphic images of $\fo$ under a homomorphism that sends $[x]$ into $x$. A crucial concept is the notion of skeleton. A \emph{skeleton mapping} for a  regular semigroup $S$ is a mapping $\phi:\G\to S^1$ satisfying
\begin{itemize}
\item[$(i)$] $x\phi\in S$, $x'\phi\in V(x\phi)$, $g_{xx'}\phi=(x\phi)(x'\phi)$ and $1\phi\in E(S^1)$;
\item[$(ii)$] $(1\phi)(x\phi)=x\phi=(x\phi)(1\phi)$ and $(1\phi)(x'\phi)=x'\phi=(x'\phi)(1\phi)$;
\item[$(iii)$] $g\phi\in S\big(g^{\phi,r},g^{\phi,l}\big)$ for all $g\in \G_i$ with $i\geq 2$, where
$$g^{\phi,l}=(g^c\phi)((g^{la})'\phi)(g^l\phi)(g^{la}\phi)$$ 
and
$$g^{\phi,r}=((g^{ra})'\phi)(g^r\phi)(g^{ra}\phi)(g^c\phi)\,.$$
\end{itemize}
In this definition we are using the non-idempotent version of the definition of sandwich set since we cannot guarantee at this point that $g^{\phi,l}$ and $g^{\phi,r}$ are idempotents of $S$. However, after the following observations, we will prove that these elements are indeed idempotents of $S$. From that point on, we will consider here also the usual definition of sandwich set for idempotents.

It is easy to see that $(1\phi)(g\phi)=g\phi=(g\phi)(1\phi)$ for all $g\in\G$ by induction on $\up(g)$. In fact, by $(iii)$, for all $g\in\G_i$ with $i\geq 2$, we just need to observe that 
$$g\phi=g^{\phi,l}y(g\phi)=(g\phi)zg^{\phi,r}$$
for any $y\in V(g^{\phi,l})$ and $z\in V(g^{\phi,r})$, and then use the induction hypothesis. This same observation gives us that $(g^c\phi)(g\phi)=(g\phi)=(g\phi)(g^c\phi)$ since $g^c\phi$ is an idempotent of $S$. Further, $(i)$ guarantees that $(G\setminus\{1\})\phi\subseteq S$. In this paper, a \emph{skeleton} of $S$ is the image of $\G\setminus\{1\}$ under some skeleton mapping. Thus, except for $x\phi$ and $x'\phi$, which are mutually inverse elements with $g_{xx'}\phi=(x\phi)(x'\phi)$, all other elements of a skeleton are idempotents of $S$.

\begin{lem}\label{l51}
Let $S$ be a regular semigroup and $\phi:\G\to S^1$ be a skeleton mapping. Then $g^{\phi,l}$ and $g^{\phi,r}$ are idempotents of $S$ for all $g\in\G_i$ with $i\geq 2$.
\end{lem}

\begin{proof}
We need to prove that $g^{\phi,l}$ and $g^{\phi,r}$ are idempotents of $S$ simultaneously by induction on $i$. For $i=2$, note that $g^l=g_{xx'}=g^r$ and $g^c=1$. Thus, we can easily check that $g^{\phi,l}$ and $g^{\phi,r}$ belong to the set $\{aa',a'a\}$ for $a=x\phi$ and $a'=x'\phi$; whence they are idempotents.

Assume now that $i>2$. By definition of skeleton mapping, 
$$g^l\phi\in S\big((g^l)^{\phi,r},(g^l)^{\phi,l}\big)\,,$$
and we may assume that $(g^l)^{\phi,l}$ and $(g^l)^{\phi,r}$ are idempotents by the induction hypothesis. Thus $(g^l)^{\phi,l}(g^l\phi)=g^l\phi$ and $(g^l\phi)(g^l)^{\phi,r}=g^l\phi$. Note also that $g^c=g^{l^2}$ and $g^{la}=(g^{l^2a})'$, or $g^c=g^{lr}$ and $g^{la}=(g^{lra})'$. If $g^c=g^{l^2}$ and $g^{la}=(g^{l^2a})'$, then 
$$\begin{array}{ll}
(g^l\phi)(g^{la}\phi)(g^c\phi)((g^{la})'\phi)(g^l\phi) \hspace*{-.1cm}& =\;(g^l\phi)(g^{lc}\phi)(g^{la}\phi)(g^c\phi)((g^{la})'\phi)(g^l\phi) \\ [.2cm]
&=\;(g^l\phi)(g^{lc}\phi)((g^{l^2a})'\phi)(g^{l^2}\phi)(g^{l^2a}\phi)(g^l\phi) \\ [.2cm]
&=\;(g^{l}\phi)((g^l)^{\phi,l})(g^l\phi)\;=\;g^l\phi\,.
\end{array}$$
If $g^c=g^{lr}$ and $g^{la}=(g^{lra})'$, then 
$$\begin{array}{ll}
	(g^l\phi)(g^{la}\phi)(g^c\phi)((g^{la})'\phi)(g^l\phi) \hspace*{-.1cm}& =\;(g^l\phi)(g^{la}\phi)(g^c\phi)((g^{la})'\phi)(g^{lc}\phi)(g^l\phi) \\ [.2cm]
	&=\;(g^l\phi)((g^{lra})'\phi)(g^{lr}\phi)(g^{lra}\phi)(g^{lc}\phi)(g^l\phi) \\ [.2cm]
	&=\;(g^{l}\phi)((g^l)^{\phi,r})(g^l\phi)\;=\;g^l\phi\,.
\end{array}$$
It is easy to see now that $g^{\phi,l}=(g^c\phi)((g^{la})'\phi)(g^l\phi)(g^{la}\phi)$ is an idempotent of $S$. The proof that $g^{\phi,r}$ is also an idempotent of $S$ is similarly.
\end{proof}

If $a$ is an element of a regular semigroup $S$, we can easily construct recursively a skeleton mapping such that $x\phi=a$. We can choose an inverse $a'$ of $a$ in $S$ and define $x'\phi=a'$, $g_{xx'}\phi=aa'$ and $1\phi=1\in S^1$. Then, assuming that $g\in\G_i$ with $i\geq 2$ and $\phi$ is defined for all $h\in\G_j$ with $j<i$, we choose
$$y\in S\big(g^{\phi,r},g^{\phi,l}\big)\,,$$
and set $g\phi=y$. Note that $S\big(g^{\phi,r},g^{\phi,l}\big)\neq\emptyset$ since $S$ is regular. We cannot however guarantee that the skeleton mapping $\phi$ is unique. In terms of skeletons, each element $a\in S$ belongs to at least one skeleton of $S$, but it may belong to many distinct skeletons.

In the next result we show that every skeleton mapping for a regular semigroup $S$ can be uniquely extended into a homomorphism from $\foo$ to $S^1$.

\begin{prop}\label{r52}
If $\phi:\G\to S^1$ is a skeleton mapping, then there is a unique (semigroup) homomorphism $\varphi:\foo\to S^1$ extending $\phi$ (that is, such that $g\phi=[g]\varphi$ for all $g\in\G$). Furthermore, $\langle\G\phi\rangle=\foo\varphi$ is a regular subsemigroup of $S^1$ and a monoid with identity element $1\phi$.
\end{prop}

\begin{proof}
Since $\Gp$ is freely generated by $\G$, let $\phi_1:\Gp\to S^1$ be the only homomorphism that extends $\phi$. By definition of skeleton mapping,
$$\rho_e\cup\left\{(g^cg^Lg,g),(gg^Rg^c,g)|\, g\in \G \mbox{ with }\up(g)\geq 2\right\}$$
is contained in the kernel of $\phi_1$. To conclude that $\rho$ is contained in the kernel of $\phi_1$, it is now enough to show that $(g^r\!\!\cdot\! g\!\cdot\! g^l,g^r\!\!\cdot\! g^c\!\cdot\! g^l)$ belongs to that kernel for all $g\in\G_i$ with $i\geq 2$.

Due to space constrains, we denote by $w$ and $z$ the expressions
$$(g^l\phi)(g^{la}\phi)(g^c\phi)((g^{la})'\phi)(g^l\phi)\quad\mbox{ and }\quad (g^r\phi)(g^{ra}\phi)(g^c\phi)((g^{ra})'\phi)(g^r\phi)\,,$$
respectively. In the proof of Proposition \ref{l51} we showed that $g^l\phi=w$. Similar arguments allow us to conclude that $g^r\phi=z$. We need also the equality $g\phi=(g^c\phi)(g\phi)(g^c\phi)$ that follows from an observation made prior to Lemma \ref{l51}. So,
$$\begin{array}{ll}
(g^r\!\!\cdot\! g\!\cdot\! g^l)\phi_1\hspace*{-.1cm}&=\;(g^r\phi)(g^{ra}\phi)(g\phi)((g^{la})'\phi)(g^l\phi) \\ [.2cm]
&=\;z\,(g^{ra}\phi)(g^c\phi)(g\phi)(g^c\phi)((g^{la})'\phi)\,w \\ [.2cm]
&=\;(g^r\phi)(g^{ra}\phi)(g^c\phi)g^{\phi,r}(g\phi)g^{\phi,l}(g^c\phi)((g^{la})'\phi)(g^l\phi) \\ [.2cm]
&=\;(g^r\phi)(g^{ra}\phi)(g^c\phi)g^{\phi,r}g^{\phi,l}(g^c\phi)((g^{la})'\phi)(g^l\phi) \\ [.2cm]
&=\;z\,(g^{ra}\phi)(g^c\phi)((g^{la})'\phi)\,w \\ [.2cm]
&=\;(g^r\phi)(g^{ra}\phi)(g^c\phi)((g^{la})'\phi)(g^l\phi)\;=\;(g^r\!\!\cdot\! g^c\!\cdot\! g^l)\phi_1\,,
\end{array}$$
where the fourth ``='' sign is due to $g\phi\in S\big(g^{\phi,r},g^{\phi,l}\big)$.

So, we have shown that $\rho$ is contained in the kernel of $\phi_1$. Thus, if $\varphi_1:\Gp\to\foo$ denotes the natural quotient homomorphism, then there exists a homomorphism $\varphi:\foo\to S^1$ such that $\phi_1=\varphi_1\varphi$. Hence $\varphi$ extends $\phi$. If $\varphi':\foo\to S^1$ is another homomorphism extending $\phi$, then $\varphi_1\varphi'=\phi_1=\varphi_1\varphi$. Consequently $\varphi'=\varphi$ since $\varphi_1$ is surjective. We proved there exists a unique homomorphism $\varphi:\foo\to S^1$ extending $\phi$.

Finally,  the second part of this result follows immediately from the fact that $\foo$ is a regular monoid with identity element [1] (this property is inherited by homomorphic images) and it is generated by $\{[g]\,|\;g\in\G\}$.
\end{proof}

Note that, in the previous proposition, $\foo\varphi$ may not be a submonoid of $S^1$ since $[1]\varphi$ may not be the identity of element of $S^1$. Given a skeleton mapping $\phi$, we can only guarantee that $1\phi$ is an identity element for the subsemigroup $\foo\varphi$. Also by the previous result, a skeleton of $S$ generates always a regular subsemigroup of $S$. We leave this conclusion registered in the following corollary for future reference.

\begin{cor}\label{r53}
If $A$ is a skeleton of $S$, then $\langle A\rangle$ is a regular subsemigroup of $S$.
\end{cor}

Let $T$ be a regular subsemigroup of a (not necessarily regular) semigroup $S$. Next, we prove that if $T$ is weakly generated by an element $a$, then there exists a homomorphism $\varphi:\fo\to S$ such that $[x]\varphi=a$ and $\fo\varphi=T$.

\begin{cor}
Let $T$ be a regular subsemigroup of a semigroup $S$. If $T$ is weakly generated by $a$, then $T$ is the image of $\fo$ under some homomorphism $\varphi:\fo\to S$ such that $[x]\varphi=a$.
\end{cor}

\begin{proof}
Let $\phi:\G\to T^1$ be a skeleton mapping such that $x\phi=a$. We have already observed that $\phi$ exists. Let $\varphi:\foo\to T^1$ be the unique homomorphism extending $\phi$ given by Proposition \ref{r52}. Note that $\varphi_{|\fo}$ is a homomorphism from $\fo$ to $T$. Thus $\fo\varphi$ is a regular subsemigroup of $T$ containing $a$. Since $T$ is weakly generated by $a$, we conclude that $\fo\varphi=T$. Finally, note that we can consider $\varphi$ as a homomorphism from $\fo$ to $S$, and $[x]\varphi=a$ as wanted.
\end{proof}

The main result of this paper is now an obvious consequence of the previous corollary.

\begin{teor}\label{fiwig}
All regular semigroups weakly generated by an element $a$ are homomorphic images of $\fo$ under a homomorphism that sends $[x]$ into $a$.
\end{teor}

The reverse of this theorem is not true however, that is, there are homomorphic images of $\fo$ that are not weakly generated by the image of $[x]$. Next, we present one such example. This example is adapted from the one used in \cite{LO22} to show that not all homomorphic images of $\ftt$ are weakly generated by two idempotents.\vspace*{.3cm}

\noindent {\bf Example 1}: We begin by defining a presentation $\langle X,R_1\rangle$ for a semigroup $S_1$. We set $X=\{x,x',g_1,g_2,h,0\}$ and include in $R_1$ the relations corresponding to the information $e0=0=0e$ for all $e\in X$, $x'\in V(x)$, $g_1\in S(x'x,xx')$, $g_2\in S(xx',x'x)$ and $h\in S(xg_2x',xg_1x')$. So $0$ correspond to a zero element of $S_1$ as expected. Now, add to $R_1$ also the pairs $(u,0)$ for 
$$u\in\{g_ig_j,\;g_iag_j\,|\;\{i,j\}=\{1,2\}\mbox{ and }a\in\{x,x'\}\}\cup\{x^2,(x')^2\}\,.$$
The semigroup $S_1$ is the one obtained from $\fo$ by collapsing all $\Dc$-classes into the $0$ element except for $\Dcc_{[x]}$, $\Dcc_{[g_{2,e,1}]}$, $\Dcc_{[g_{2,e,2}]}$ and $\Dcc_{[g_{3,d,2}]}$. In fact, the first set of relations identifies $g_1$, $g_2$ and $h$ with the elements $g_{2,e,1}$, $g_{2,e,2}$ and $g_{3,d,2}$, respectively, while the second set of relations collapses all other $\Dc$-classes into $0$. 

But the semigroup we are interested in is the semigroup $S$ given by the presentation $\langle X,R\rangle$ where $R=R_1\cup\{(g_1,g_1x'hxg_1),(g_2,g_2x'hxg_2)\}\,$. Since $R$ contains $R_1$, the semigroup $S$ is also a homomorphic image of $\fo$. The semigroup $S$ is, in fact, obtained from $S_1$ by merging the $\Dc$-classes of $g_1$ and $g_2$ into distinct `blocks' of the $\Dc$-class of $h$. We illustrate the ``Egg-box'' diagram of the $\Dc$-classes of $S$ in Figure \ref{figT1} (note that each congruence class is represented by one of its elements; we write also $e=xx'$, $f=x'x$ and $h_1=x'hx$ due to space constrains).

\begin{figure}[ht]
\begin{tabular}{c}
\begin{tabular}{|c|c|}
\hline $x$ & $e\;*$\\ 
\hline $f\; *$ & $x'$ \\ \hline 
\end{tabular} \\ [.6cm]			
\begingroup
\renewcommand{\arraystretch}{1.2}
{\small\begin{tabular}{|c|c|c|c|c|c|c|c|}
\hline	$g_1e\;*$ & $g_1x$ & $g_1x'$ & $g_1\;*$ & $g_1x'h\;*$ & $g_1h_1$ & $g_1h_1g_2$ & $g_1h_1g_2x$ \\ 
\hline	$x'g_1e$ & $x'g_1x\; *$ & $x'g_1x'$ & $x'g_1$ & $x'g_1x'h$ & $\!x'g_1h_1\;*\!$ & $\!x'g_1h_1g_2\!$ & $\!x'g_1h_1g_2x\!$ \\ 
\hline 	$x^2x'$ & $x^2$ & $xg_1x'\;*$ & $xg_1$ & $h\;*$ & $hx$ & $hxg_2$ & $hxg_2x$ \\
\hline $fe$ & $fx$ & $fg_1x'$ & $fg_1\;*$ & $x'h$ & $h_1\;*$ & $h_1g_2\;*$ & $h_1g_2x$ \\
\hline $xg_2fe$ & $xg_2fx$ & $xg_2fg_1x'$ & $xg_2fg_1$ & $xg_2x'\;*$ & $g_2f$ & $xg_2$ & $xg_2x$ \\			
\hline $g_2fe$ & $g_2fx$ & $g_2fg_1x'$ & $g_2fg_1$ & $g_2x'$ & $g_2f\;*$ & $g_2\;*$ & $g_2x$ \\	
\hline $efe$ & $efx$ & $efxg_1x'$ & $efg_1$ & $ex'$ & $ef$ & $eg_2\;*$ & $eg_2x$ \\
\hline $x'fe$ & $x'fx$ & $x'fg_1x'$ & $x'fg_1$ & $(x')^2$ & $x'f$ & $x'g_2 $ & $x'g_2x\;*$ \\	
\hline
\end{tabular}}
\endgroup\\ [2.2cm]
\begin{tabular}{|c|}
\hline	$0\;*$\\ \hline
\end{tabular}
\end{tabular}
\caption{``Egg-box'' diagram of $S$\hspace*{.2cm} {\small($*$ - idempotent)}.}\label{figT1}
\end{figure}

The Semigroup $S$ is a homomorphic image of $\fo$ but it is not weakly generated by $x$. The subsemigroup $T$ generated by the set $\{x,x',g_1x'hxg_2\}$ is a proper regular subsemigroup of $S$. It is obtained from $S$ by deleting the fifth and sixth rows and the third and fourth columns of $\Dcc_{h}$. 
\hfill\qed\vspace*{.3cm}
 
The fact that we were able to adapt the example for $\ftt$ to an example for $\fo$ shouldn't be so surprising. In the next section we show that $\ftt$ can be embedded into $\fo$ in a natural way. But before, let us end this section by proving that $\fo$ is unique up to isomorphism.

\begin{prop}
Let $T$ be a regular semigroup weakly generated by $x$ such that all other regular semigroups weakly generated by $x$ are homomorphic images of $T$ (under a homomorphism that fixes $x$). Then $T$ and $\fo$ are isomorphic.
\end{prop}

\begin{proof}
By the universal property of both $T$ and $\fo$, there are surjective homomorphisms
$\varphi:\fo\to T$ and $\psi:T\to\fo$ such that $([x])\varphi=x$ and $x\psi=[x]$. Thus $\varphi\circ\psi$ is an endomorphism of $\fo$ such that $([x])\varphi\circ\psi=[x]$. By Lemma \ref{sandwich}, the restriction of $\varphi\circ\psi$ to the set $\{[g]\,|\;g\in\G\setminus\{1\}\}$ must be the identity mapping. Therefore, $\varphi\circ\psi$ is the identity automorphism of $\fo$ since this semigroup is generated by the previous set. We have shown that $\varphi$ and $\psi$ are mutually inverse isomorphisms, whence $T$ and $\fo$ are isomorphic.
\end{proof}

\section{Embedding $\ftt$ into $\fo$}\label{sec6}

In this section we prove that $\fo$ has a regular subsemigroup $\fot$ weakly generated by $\{[xx'],[x'x]\}$ isomorphic to $\ftt$. The first step is to construct $\fot$. After, we prove that $\fot$ is isomorphic to $\ftt$. It is also important that the reader recalls the notation used in Section \ref{sec3} for the elements of $\G_{2,e}$ (page \pageref{G2e}) and for some of the elements of $\G_{3,d}$ (page \pageref{G3d}). 

\subsection{The construction of $\fot$}

We begin by defining a subset $\G^\circ=\cup_{i\in\mathbb{N}}\G_i^\circ$ of $\G^5$. We set 
$$\G_1^\circ=\G_1=\{g_{xx'}\}\quad\mbox{ and }\quad \G_2^\circ=\G_2=\G_{2,e} =\{g_{2,e,1},g_{2,e,2}\}\,.$$ 
For $i>2$, $\G_i^\circ$ is always a proper subset of $\G_{i,d}\,$:
$$\G_3^\circ=\{g_{3,d,1},\;g_{3,d,2},\;g_{3,d,3},\;g_{3,d,4}\}\subset\G_{3,d}$$
and 
$$\G_{i+1}^\circ=\{g\in\G_{i+1,d}\,|\;g^l,g^r\in\G_i^\circ\}\,.$$
Hence $\G^\circ\subset \G_1\cup\G_2\cup\G_d\,$. The next result is obvious by induction on the indexes $i$. We omit its proof.

\begin{lem}\label{l61}
If $g\in\G^\circ$, then $\ep(g)\subseteq\G^\circ$.
\end{lem}

Let $\mrx^\circ$ be the set of all nontrivial mountain ranges 
$$u=g_0a_1g_1\cdots g_{2n-1}a_{2n}g_{2n}$$
such that:
\begin{itemize}
\item[$(a)$] $g_0g_1\cdots g_{2n}\in (\G^\circ\cup\{1\})^+$; and
\item[$(b)$] if $\up(g_i)=1$, then $a_{i+1}=(a_i)'$.
\end{itemize} 
The next result shows us that $\mrx^\circ$ is closed for uplifting of rivers.

\begin{lem}\label{l62}
Let $u=g_0a_1g_1\cdots g_{2n-1}a_{2n}g_{2n}\in\mrx^\circ$ with a river $g_i$. If $v$ is such that $u\xr{g_i}v$, then $v\in\mrx^\circ$.
\end{lem}

\begin{proof}
Assume first that $g_{i-1}=g_{i+1}$ and $a_i=a_{i+1}'$. Then 
$$v=g_0\cdots g_{i-1}a_{i+2}g_{i+2}\cdots g_{2n}\,.$$
Clearly $(a)$ is satisfied, and if $\up(g_{i-1})\neq 1$, then also $(b)$ is satisfied. If $\up(g_{i-1})=1$, then $\up(g_{i+1})=1$, and by condition $(b)$ (applied to $u$) we know that $a_{i-1}=a_i'$ and $a_{i+1}=a_{i+2}'$. Thus $a_{i-1}=a_{i+2}'$ since $a_i=a_{i+1}'$, and $v$ continues to satisfy $(b)$ even in this case. In conclusion, $v\in\mrx^\circ$ if $g_{i-1}=g_{i+1}$ and $a_i=a_{i+1}'$.

Assume now that $g_{i-1}\neq g_{i+1}$ or $a_i\neq a_{i+1}'$. Then $$h_i=(g_{i+1},a_{i+1}',g_i,a_i,g_{i-1})\in\G$$
and 
$$v=g_0\cdots g_{i-1}a_ih_ia_{i+1}g_{i+1}\cdots g_{2n}\,.$$
In this case, $v$ clearly satisfies $(b)$ since the anchors subsequence does not change and $\up(h_i)\neq 1$. If $g_i=1$, then $h_i\in\G_2=\G_2^\circ$ and $v$ satisfies $(a)$ also. If $g_i=g_{xx'}$, then $a_i=a_{i+1}'$ by condition $(b)$ (applied to $u$) and so $g_{i-1}\neq g_{i+1}$. Consequently, $h_i\in\G_3^\circ$ and $v$ satisfies $(a)$. Finally, if $\up(g_i)\geq 2$, then $g_{i-1},g_{i+1}\in\G_d$. If $g_{i-1}=g_{i+1}$, then also $a_{i}=a_{i+1}'$ since $g_{i-1}^l\neq g_{i-1}^r$, which is a contradiction. Therefore, $g_{i-1}\neq g_{i+1}$ and, once more, $v$ satisfies $(a)$ since $h_i\in\G^\circ$. Note that we have shown that $v\in\mrx^\circ$ if $g_{i-1}\neq g_{i+1}$ or $a_i\neq a_{i+1}'$. Thus, we have ended the proof of this result.
\end{proof}

We denote by $\mx^\circ$ the set of all mountains from $\mrx^\circ$. If $u\in\mx^\circ$, then $(a)$ is equivalent to $\ka(u)=g_n\in\G^\circ$ by Lemma \ref{l61}; while $(b)$ is equivalent to the anchors subsequence $a_1a_2\cdots a_{2n-1}a_{2n}$ having a prefix and a suffix from the set $\{11,x'x\}$. In particular, the only two mountains of $\mx^\circ$ of height $1$ are $\be_1(xx')=11g_{xx'}11$ and $\be_1(x'x)=1x'g_{xx'}x1$. In the next result we prove that $\mx^\circ$ is a regular subsemigroup of $\mx$.

\begin{prop}
$\mx^\circ$ is a regular subsemigroup of $\mx$.
\end{prop}

\begin{proof}
Let $u,v\in\mx^\circ$. By Lemma \ref{l62}, $u\odot v$ is a mountain from $\mrx^\circ$ and $\mx^\circ$ is a subsemigroup of $\mx$. Now, $\mx^\circ$ is regular because $\cev{u}\in\mx^\circ$ for every $u\in\mx^\circ$, and $\cev{u}$ is an inverse of $u$.
\end{proof}

We denote by $\fot$ the regular subsemigroup of $\fo$ corresponding to $\mx^\circ$. In the next subsection we prove that $\fot$ is isomorphic to $\ftt$.

\subsection{An isomorphism from $\ftt$ to $\fot$} It is important for this subsection that the reader recalls the construction of $\ftt$ described in Section \ref{sec2}. It should be clear at this point the connection between the theory developed here and the one developed in \cite{LO22}. Roughly speaking, we can look to the theory developed in \cite{LO22} as the simplified version for triples of the theory developed here, obtained by `dropping' the information given by the anchors. So, the meaning of notions like \emph{i-landscape} and \emph{i-mountain range} (called landscape and mountain range in \cite{LO22}) and of operations like uplifting of i-rivers should now be clear to the reader. Otherwise, we advice the reader to look for more information in \cite{LO22} before she/he proceeds since a clear understanding of these notions is important for what follows.

We begin by defining recursively a mapping
$$\varphi:\G^\circ\setminus\{g_{xx'}\}\to\Ho\setminus\{1,e,f\}\,.$$ 
We start by setting
$$(g_{2,e,i})\varphi=h_{2,i}\quad\mbox{ and }\quad (g_{3,d,j})\varphi=h_{3,j}\,,$$
for $i\in\{1,2\}$ and $j\in\{1,2,3,4\}$. Note that if $g_{3,d,j}=(g,a,g_{xx'},a,g_1)$, for $\G_2^\circ=\{g,g_1\}$ and $a\in\{1,x'\}$, then
$$(g_{3,d,i})\varphi=(g\varphi,e_1,g_1\varphi)\,,$$
where $e_1=e$ if $a=1$, or $e_1=f$ if $a=x'$. Extend $\varphi$ to all $g\in\G_i^\circ$ with $i>3$ by setting
$$g\varphi=(g^l\varphi,g^c\varphi,g^r\varphi)\,.$$
By the recursive definition, if $g^{ls}=g^c=g^{rt}$ for some $s,t\in\{l,r\}$, then $(g^l\varphi)^s=g^c\varphi=(g^r\varphi)^t$; whence $g\varphi\in\Ho$ and $\varphi$ is well defined. We will conclude next that $\varphi$ is a bijection by defining its inverse mapping.

Consider now the mapping $\psi:\Ho\setminus\{1,e,f\}\to\G^\circ\setminus\{g_{xx'}\}$ obtained recursively as follows. Set
$$(h_{2,i})\psi=g_{2,e,i}\quad\mbox{ and }\quad (h_{3,j})\psi=g_{3,d,j}$$
for $i\in\{1,2\}$ and $j\in\{1,2,3,4\}$. Observe that if $h_{3,j}=(h,e_1,h_1)$, for $\Ho_2=\{h,h_1\}$ and $e_1\in\Ho_1=\{e,f\}$, then
$$(h_{3,j})\psi=(h\psi,a,g_{xx'},a,h_1\psi)\,,$$
where $a=1$ if $e_1=e$, or $a=x'$ if $e_1=f$. Let now $h\in\Ho_i$ with $i>3$ and define recursively
$$h\psi=(h^l\psi,a,h^c\psi,b,h^r\psi)\,,$$
where $a'=(h^l\psi)^{sa}$ for $h^{ls}=h^c$ and $b'=(h^r\psi)^{ta}$ for $h^{rt}=h^c$. By the recursive definition and the choice of $a$ and $b$, the 5-tuple $h\psi$ belongs to $\G^\circ$ and $\psi$ is well defined.

Clearly $\varphi_{|\G_k^\circ}$ and $\psi_{|\Ho_k}$ are mutually inverse bijections between $\G_k^\circ$ and $\Ho_k$ for $k\in\{2,3\}$. If we now assume that  $\varphi_{|\G_j^\circ}$ and $\psi_{|\Ho_j}$ are mutually inverse bijections between $\G_j^\circ$ and $\Ho_j$ for all $j<i$, we can immediately conclude that $\varphi_{|\G_i^\circ}$ and $\psi_{|\Ho_i}$ are mutually inverse bijections between $\G_i^\circ$ and $\Ho_i$ too. We just need to observe that  the entries $g^l$, $g^c$ and $g^r$ of $g$ completely determine the other entries $g^{la}$ and $g^{ra}$ for all $g\in\G_i^\circ$. Therefore, we can conclude by induction that $\varphi_{|\G_i^\circ}$ and $\psi_{|\Ho_i}$ are mutually inverse bijections between $\G_i^\circ$ and $\Ho_i$ for all $i\geq 2$. Thus $\varphi$ and $\psi$ are also mutually inverse bijections between $\G^\circ\setminus\{g_{xx'}\}$ and $\Ho\setminus\{1,e,f\}$. Note further that $\varphi$ and $\psi$ have the following property:
$$(g^s)\varphi=(g\varphi)^s\quad\mbox{ and }\quad (h^t)\psi=(h\psi)^t$$
for any $s,t\in\{l,c,r\}$.

Denote by $\mri$ and $\mi$ the sets of all nontrivial i-mountain ranges and i-mountains, respectively, from $\Ho^+$. Next, we define two mutually inverse bijections $\ol{\varphi}$ and $\ol{\psi}$ between $\mrx^\circ$ and $\mri$ by applying naturally the mappings $\varphi$ and $\psi$ to each letter. There is however a problem with this approach: both $\varphi$ and $\psi$ are not yet defined for letters of height less than $2$, and their definition for letters of height 1 is not so obvious since $\G_1^\circ$ has only one letter, namely $g_{xx'}$, while $\Ho_1^\circ$ has two letters, namely $e$ and $f$.

Let $\G^\circ\!'=(\G^\circ\setminus\{g_{xx'}\})\cup\{1,1g_{xx'}1,x'g_{xx'}x\}\,$ and extend $\varphi$ to a bijection from $\G^\circ\!'$ onto $\Ho$ by setting
$$1\varphi=1,\quad (1g_{xx'}1)\varphi=e,\quad (x'g_{xx'}x)\varphi=f\,.$$
If $u=g_0a_1g_1\cdots g_{2n-1}a_{2n}g_{2n}\in\mrx^\circ$, define
$$u\ol{\varphi}=h_0h_1\cdots h_{2n-1}h_{2n}\,,$$
where $h_i=(a_ig_{xx'}a_{i+1})\varphi$ for $g_i=g_{xx'}$, and $h_i=g_i\varphi$ for the other cases. The mapping $\ol{\varphi}$ is well defined since $a_ia_{i+1}\in\{11,x'x\}$ if $g_i=g_{xx'}$ by condition $(b)$ of the definition of $\mrx^\circ$. Also by definition of $\varphi$, the word $u\ol{\varphi}$ is clearly a nontrivial i-landscape; whence $\ol{\varphi}$ is a mapping from $\mrx^\circ$ to $\mri$.

Extend now $\psi$ to the inverse mapping of $\varphi:\G^\circ\!'\to\Ho$ and let $v=h_0h_1\cdots h_{2n-1}h_{2n}$ be a nontrivial i-mountain range. Set $g_i=g_{xx'}$ if $h_i\in\{e,f\}$, or $g_i=h_i\psi$ otherwise. If $h_i=e$ [$h_i=f$], set also $a_i=1=a_{i+1}$ [$a_i=x'$ and $a_{i+1}=x$]; whence $h_i\psi=a_ig_ia_{i+1}$ if $h_i\in\{e,f\}$. If $\up(h_{i-1})\geq 2$ and $\up(h_i)\geq 2$, then only one of the letters $g_{i-1}$ or $g_i$ is the left or the right entry of the other (and not both entries). Hence, there is only one anchor $a_i$ that turns the triplet $g_{i-1}a_ig_i$ anchored. Define now
$$v\ol{\psi}=g_0a_1g_1\cdots g_{2n-1}a_{2n}g_{2n}\,.$$
By the previous observations and the definition of $\psi$, $v\ol{\psi}$ is a mountain range from $\mrx^\circ$. A close inspection allows us to conclude also that $\ol{\varphi}$ and $\ol{\psi}$ are inverse mappings since $\varphi$ and $\psi$ are inverse bijections too. 

Note that both $u\in\mrx^\circ$ and $u\varphi$ have the same line graph structure: they differ only on the labels. In particular, $\ol{\varphi}_{|\mx^\circ}$ and $\ol{\psi}_{|\mi}$ are mutually inverse bijections between the sets $\mx^\circ$ and $\mi$. We leave these observations registered in the following lemma for future reference.

\begin{lem}
The mappings $\ol{\varphi}:\mrx^\circ\to\mri$ and $\ol{\psi}:\mri\to\mrx^\circ$ are mutually inverse. Further, their restrictions $\ol{\varphi}_{|\mx^\circ}$ and $\ol{\psi}_{|\mi}$ to the sets of (nontrivial) mountains from $\mrx^\circ$ and (nontrivial) i-mountains, respectively, give us mutually inverse bijections between these sets.
\end{lem}

It is not so hard to show now that $\ol{\varphi}_{|\mx^\circ}$ is an isomorphism. In fact, if we look carefully to the proof of Lemma \ref{l62}, we see that if
$$u=g_0a_1g_1\cdots g_{2n-1}a_{2n}g_{2n}\in\mrx^\circ\,,\quad u\ol{\varphi}=h_0h_1\cdots h_{2n-1}h_{2n}\,,$$
and $u\xr{g_i}v$ for some river $g_i$, then $u\ol{\varphi}\xr{h_i}v\ol{\varphi}\,$. Thus $$(u_1\odot u_2)\ol{\varphi}_{|\mx^\circ}=(u_1\ol{\varphi}_{|\mx^\circ})\odot (u_2\ol{\varphi}_{|\mx^\circ})$$
by recursion. We have proved the following result:

\begin{prop}
$\ol{\varphi}_{|\mx^\circ}:\mx^\circ\to\mi$ is an isomorphism with inverse isomorphism $\ol{\psi}_{|\mi}$.
\end{prop}

\begin{cor}\label{c66}
$\fot$ is a regular subsemigroup of $\fo$ weakly generated by the set of idempotents $\{[xx'],[x'x]\}$ and isomorphic to $\ftt$.
\end{cor}

\begin{proof}
Since $\fot$ is isomorphic to $\mx^\circ$ and $\ftt$ is isomorphic to $\mi$, then $\fot$ is isomorphic to $\ftt$ by the previous proposition. In fact, the isomorphism $\ol{\psi}_{|\mi}$ induces an isomorphism $\phi:\ftt\to\fot$ such that $(e\varrho)\phi=[xx']$ and $(f\varrho)\phi=[x'x]$. Since $\ftt$ is weakly generated by the set $\{e\varrho,f\varrho\}$, then $\fot$ is weakly generated by $\{[xx'],[x'x]\}$. 
\end{proof}

Since $\phi$ is an isomorphism, we know from \cite{LO22} that $\fot$ is generated by the set $((\Ho\setminus\{1\})\varrho)\phi$. At first sight one may be lead to believe that $((\Ho\setminus\{1\})\varrho)\phi=\{[g]\,|\;g\in\G^\circ\}\cup\{[x'x]\}$ since $\Ho\psi= (\G^\circ\setminus \{g_{xx'}\})\cup\{1,1g_{xx'}1,x'g_{xx'}x\}$; and therefore one may be lead to think that $\fot$ is generated by the set $\{[g]\,|\;g\in\G^\circ\}\cup\{[x'x]\}$. However, this is far from being true. In fact, we cannot even guarantee that all $[g]$ belong to $\fot$ for $g\in\G^\circ$. For example, if we take 
$$g_{3,d,2}=(g_{2,e,1},x',g_{xx'},x,g_{2,e,2})\in\G_3^\circ\,,$$
then $\be(g_{3,d,2})=\be_1(g_{3,d,2})$ is the mountain
$$11g_{xx'}xg_{2,e,1}x'g_{3,d,2}xg_{2,e,2}x'g_{xx'}11\,,$$
which does not belong to $\mx^\circ$ since it fails to satisfy the condition $(b)$. Hence $[g_{3,d,2}]\not\in\fot$. Note that, in this case, $[x'g_{3,d,2}x]\in\fot$ as
$$\be(x'g_{3,d,2}x)=1x'g_{xx'}xg_{2,e,1}x'g_{3,d,2}xg_{2,e,2}x'g_{xx'}x1\in\mx^\circ\,.$$

The example of the previous paragraph somehow illustrates what happens in general. A deeper analyzes into which $[g]$ belong to $\fot$, for $g\in\G^\circ$ with $\up(g)\geq 2$, leads us to three distinct cases. In two of them we conclude that $[g]$ belongs to $\fot$, but in the third is the $\rho$ congruence class $[x'gx]$ that belongs to $\fot$ (and not $[g]$). The three distinct cases are the following ones:
\begin{itemize}
\item[$(i)$] If $\up(g)$ is even, then $\be_1(g)\in\mx^\circ$ and $[g]\in\fot$.
\item[$(ii)$] If $\up(g)=2n+1$, then $g^{c^{n-1}}\in\G_3^\circ$. If $g^{c^{n-1}}=g_{3,d,i}$ for $i$ odd, then $11g_{xx'}1$ is a prefix of $\be_1(g)$ and $1g_{xx'}11$ is a suffix of $\be_1(g)$. We can now see that $\be_1(g)$ also belongs to $\mx^\circ$, and so $[g]\in\fot$.
\item[$(iii)$] If $\up(g)=2n+1$ with $g^{c^{n-1}}=g_{3,d,i}$ for $i$ even, then $11g_{xx'}x$ is a prefix of $\be_1(g)$ and $x'g_{xx'}11$ is a suffix of $\be_1(g)$. Hence $\be_1(g)\not\in\mx^\circ$ because it fails to satisfy the condition $(b)$; whence $[g]\not\in\fot$. However, $\be(x'gx)$ is equal to $\be(g)$ except for the first anchor, which becomes $x'$, and the last anchor, which becomes $x$. So $\be(x'gx)\in\mx^\circ$ and $[x'gx]\in\fot$.
\end{itemize}	
Let $\ol{g}=g$ for the two first cases and $\ol{g}=x'gx$ for the last one. A careful analyzes to the definition of $\phi$ allows us to conclude that
$$((\Ho\setminus\{1\})\varrho)\phi=\{[xx'],[x'x]\}\cup\{[\ol{g}]\,|\; g\in\G^\circ\mbox{ with } \up(g)\geq 2\}\,.$$
Consequently, $\fot$ is generated by the previous set.

In \cite[Corollary 6.10]{LO22} we proved that each regular semigroup weakly generated by a finite set of idempotents strongly divides $\ftt$, that is, it is a homomorphic image of a regular subsemigroup of $\ftt$. Using the Corollary \ref{c66} we now get the following result:

\begin{cor}
All regular semigroups weakly generated by a finite set of idempotents strongly divide $\fo$.
\end{cor}

Note that the previous corollary applies in particular to all regular semigroup generated by a finite set of idempotents and to all finite idempotent generated semigroups. In \cite[Corollary 6.11]{LO22} we concluded that all finite semigroups divide $\ftt$, that is, are homomorphic images of (not necessarily regular) subsemigroup of $\ftt$. We now have also that:

\begin{cor}
All finite semigroups divide $\fo$.
\end{cor}

\vspace*{.5cm}

\noindent{\bf Acknowledgments}: This work was partially supported by CMUP, member of LASI, which is financed by (Portuguese) national funds through FCT - Fundação para a Ciência e a Tecnologia, I.P., under the project with reference UIDB/00144/2020. The author also would like to acknowledge the importance of the GAP software \cite{gap}, and its Semigroup package \cite{mitchell}, in the research presented in this paper: the several simulations computed in GAP allowed the emergence of the pattern used to construct the set $\G$.

\end{document}